\documentclass[final]{siamltex}

\usepackage{amsmath} 
\usepackage{amsfonts,amssymb}
\usepackage{graphicx}      
\usepackage{psfrag}     
\usepackage{cite}

\newcommand{\RR}{\mathbb{R}}
\newcommand{\NN}{\mathcal{N}}
\renewcommand{\t}{\tilde}
\renewcommand{\SS}{\mathcal{S}}
\newcommand{\hh}{\textbf{h}}

\newtheorem{Rem}[theorem]{Remark}

\title{Asymptotic shallow water models for internal waves in a two-fluid system with a free surface}

\author{Vincent Duch\^ene\thanks{D\'epartement de Math\'ematiques et Applications, UMR 8553, \'Ecole normale sup\'erieure, 45 rue d’Ulm, F 75230
Paris cedex 05, France ; e-mail: {\tt vincent.duchene@ens.fr}}}
\date{\today}
\begin{document}

\maketitle
\begin{abstract}
 In this paper, we derive asymptotic models for the propagation of two and three-dimensional gravity waves at the free surface and the interface between two layers of immiscible fluids of different densities, over an uneven bottom. We assume the thickness of the upper and lower fluids to be of comparable size, and small compared to the characteristic wavelength of the system (shallow water regimes). Following a method introduced by Bona, Lannes and Saut based on the expansion of the involved Dirichlet-to-Neumann operators, we are able to give a rigorous justification of classical models for weakly and strongly nonlinear waves, as well as interesting new ones. In particular, we derive linearly well-posed systems in the so called Boussinesq/Boussinesq regime. Furthermore, we establish the consistency of the full Euler system with these models, and deduce the convergence of the solutions.
\end{abstract}

\begin{keywords} 
Internal waves, free surface, rigid lid configuration, asymptotic models, long waves, shallow water.
\end{keywords}

\begin{AMS}
76B55, 35Q35, 35C20, 35L40.
\end{AMS}

\pagestyle{myheadings}
\thispagestyle{plain}
\markboth{V. DUCHENE}{}

\section{Introduction}
\label{Section1}
\subsection{General settings}

This paper deals with weakly and strongly nonlinear internal waves in a two-fluid system. We consider the case of uneven bottom topography and free surface, though the rigid lid assumption is mentioned. The idealized system studied here consists in two layers of immiscible, homogeneous, ideal, incompressible and irrotationnal fluids under the only influence of gravity.

The mathematical theory of internal waves, following the theory of free-surface water waves, has attracted lots of interest over the past decades. We let the reader refer to the survey article of Helfrich and Melville~\cite{HeM} for a good overview of the ins and outs on this problem. The governing equations, that we call full Euler system, are fully nonlinear, and their direct study and computation remains a real obstacle. In particular, the well-posedness of the equations in Sobolev space is challenging, as discussed in Remark~\ref{RemSurfTension}. An alternative way is to look for approximations through the use of asymptotic models. Such models can be derived from the full Euler system by introducing natural dimensionless parameters of the system, and setting some smallness hypotheses on these parameters (thus reducing the framework to more specific physical regimes).

Many models for a two-fluid system have already been derived and studied. Systems under the rigid-lid assumption have first been investigated (see~\cite{Miyata} or~\cite{Mal} for example). Weakly nonlinear models in the free-surface case have been presented by Camassa and Choi~\cite{CCa2}. Nguyen and Dias~\cite{NgD} presented a great deal of numeric simulations for such Boussinesq-type systems. Strongly nonlinear regimes have been derived by Matsuno~\cite{Mat} and Camassa and Choi~\cite{CCa}, and Barros, Gavrilyuk and Teshukov~\cite{BGT}, generalizing the classical Green-Naghdi model (see~\cite{GN}). A different approach has been carried out by Craig, Guyenne and Kalisch~\cite{CGK}, using the Hamiltonian formulation of the Euler equations. Most of these works are formal, and restricted to two-dimensional flows, or to the flat-bottom case. Finally, we refer to the work of Bona, Lannes and Saut~\cite{BLS} who, following  a strategy initiated in~\cite{BCL,BCS1,BCS2}, rigorously derived a large class of models in different regimes, under the rigid-lid assumption. This paper is concerned with the more complex case where the rigid-lid assumption is removed and replaced by a free surface. 

The strategy consists in rewriting the full system as a system of four evolution equations located on the surface and the interface between the two fluids (as opposed to two equations in the rigid-lid case). The reformulation introduces a Dirichlet-to-Neumann operator $G[\zeta]$ and an interface operator $H[\zeta]$, defined precisely below. The computation of asymptotic expansions of these operators leads to the models presented here. We focus here on shallow water regimes, allowing strongly nonlinear waves.

Our analysis gives a rigorous derivation of most of the models existing in the literature, and also interesting new ones. In particular, we derive a set of models in the Boussinesq/Boussinesq regime, with coefficients that can be chosen so that the system is linearly well-posed. We prove that the full Euler system is consistent with each of our models, which roughly means that any solution of the full system solves the asymptotic systems up to a small error. Then in the case of the shallow water/shallow water model, using energy methods together with consistency, we also prove that the solutions of our models converge toward the solutions of the full Euler system, assuming that such solution exist.

The paper is organized as follows. Section~\ref{Section1} is devoted to the reformulation of the full system, from the Euler equation to the ``Zakharov formulation'', written in dimensionless form. In \S~\ref{linearized}, we focus on the linearized system, and its dispersion relations are derived. From the asymptotic expansion of the operators $G[\zeta]$ and $H[\zeta]$ presented in \S~\ref{asymptexp}, the asymptotic models under different regimes are rigorously obtained, and presented in \S~\ref{asymmod}. The consistency of the full Euler system with each of our models is proved. Then, \S~\ref{Solapp} gives convergence results: we show that the solutions of the full Euler system tend to associated solutions of one of our models in the shallow-water limit. Finally, the links with different models already existing in the literature are presented in \S~\ref{Section3}, for rigid-lid models~\cite{BLS} and layer-mean equations~\cite{CCa,CCa2}. The proof of Proposition~\ref{G1} is given in Appendix.

\begin{figure}[htb]
\centering
\psfrag{z1}[Bc][Br]{\begin{footnotesize}$\zeta_1(t,X)$                          \end{footnotesize}}
\psfrag{z2}[Bc][Br]{\begin{footnotesize}$\zeta_2(t,X)$                          \end{footnotesize}}
\psfrag{b}{\begin{footnotesize}$b(X)$                          \end{footnotesize}}
\psfrag{h1}[Bc][Br]{\begin{footnotesize}$h_{10}$                   \end{footnotesize}}
\psfrag{h2}[Bc][Br]{\begin{footnotesize}$-h_{20}$                   \end{footnotesize}}
\psfrag{0}[Bc][Br]{\begin{footnotesize}$0$                   \end{footnotesize}}
\psfrag{z}{\begin{footnotesize}$z$                   \end{footnotesize}}
\psfrag{X}{\begin{footnotesize}$X$                   \end{footnotesize}}
\psfrag{O1}{\begin{footnotesize}$\Omega^1_t$                   \end{footnotesize}}
\psfrag{O2}{\begin{footnotesize}$\Omega^2_t$                   \end{footnotesize}}
\psfrag{n1}{\begin{footnotesize}$n_1$                          \end{footnotesize}}
\psfrag{n2}{\begin{footnotesize}$n_2$                          \end{footnotesize}}
\psfrag{nb}[Bc][Br]{\begin{footnotesize}$n_b$                          \end{footnotesize}}
\psfrag{g}{\begin{footnotesize}$\overrightarrow{\textbf{g}}$                          \end{footnotesize}}
 \includegraphics[width=1\textwidth]{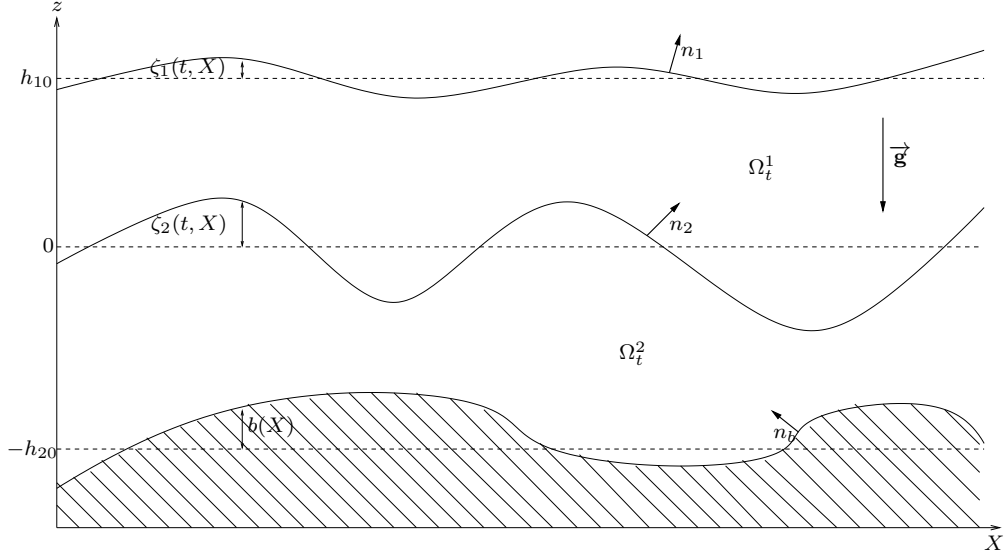}
 \caption{Sketch of the domain}
\label{Sketch}
\end{figure}

\subsection{Notation}
We use the Cartesian coordinates $(X,z)$, where $z$ is the vertical variable, and $X$ is the $d$-dimensional horizontal variable: $X=x$ when $d=1$ and $X=(x,y)$ when $d=2$. 

The symbols $\nabla$ and $\Delta$ denote the gradient and Laplace operators in the horizontal variables, respectively, whereas $\nabla_{X,z}$ and $\Delta_{X,z}$ are their $(d+1)$-variable version. For $\mu>0$, we also define the scaled version of the gradient and Laplace operators, namely ${\nabla^\mu_{X,z}:=(\sqrt\mu\nabla^T,\partial_z)^T}$ and ${\Delta^\mu_{X,z}:=\mu\Delta+\partial_z^2}$, respectively.

Given a surface $\Gamma:=\{(X,z),z=\zeta(X)\}$, we denote by $\partial_n$ the upward normal derivative at $\Gamma$:
\[\partial_n:=n\cdot\nabla_{X,z} \mbox{, with }n:=\frac{1}{\sqrt{1+|\nabla \zeta|^2}}(-\nabla\zeta,1)^T.\]
If we consider an elliptic operator $\textbf{P}=\nabla_{X,z}\cdot P\nabla_{X,z}$, then the co-normal derivative associated to $\textbf{P}$ is
\[\partial^P_n:=n\cdot P\nabla_{X,z},\]
that we simply denote $\partial_n$ when there is no risk of confusion.

For any tempered distribution $u$, we denote by $\widehat{u}$ its Fourier transform. We use the standard Fourier multiplier notation $f(D)u$, defined in terms of Fourier transforms by \[\widehat{f(D)u}:=f\widehat u.\]
The operator $\Lambda=(1-\Delta)^{1/2}$ is equivalently defined using the Fourier multiplier notation to be $\Lambda=(1+|D|^2)^{1/2}$.

We denote by $H^s(\RR^d)$ (or simply $H^s$ if the underlying domain is clear from the context) the $L^2$-based Sobolev spaces. Their norm is written $\big\vert\cdot\big\vert_{H^s}$, and simply $\big\vert\cdot\big\vert_2$ for the $L^2$ norm.

Then for $0 < T \leq \infty$, $q\in\mathbb{N}$, $W^{q,\infty}([0,T];H^s(\RR^d))$ (or simply $W^{q,\infty}H^s$, and $L^\infty H^s$ when $q=0$) denotes the space of the functions $f(t,X)$ defined on $[0,T]\times\RR^d$, whose derivative up to the order $q$ in $t$ are bounded in $H^s(\RR^d)$, uniformly with respect to $t\in [0,T)$. Their norm is written $\big\vert\cdot\big\vert_{W^{q,\infty}H^s}$.

Since it often appears, it is convenient to introduce for $s$ and $T>0$ the space $\mathbf{H}^s([0,T])$, made up of the quadruplets $(\zeta_1,\zeta_2,u_1,u_2)$ such that their components satisfy ${u_1,\ u_2 \in W^{1,\infty} ([0,T];H^{s+7/2}(\RR^d))^d}$, $\zeta_1 \in W^{1,\infty}([0,T];H^{s+3/2}(\RR^d))$ and finally ${\zeta_2 \in W^{1,\infty}([0,T];H^{s+5/2}(\RR^d))}$. Their norm is written $\big|\cdot\big|_{\mathbf{H}^s}$.

Finally, we denote by $\SS^+$ the planar strip $\RR^d\times(0,1)$, and by $\SS^-$ the planar strip $\RR^d\times(-1,0)$. We use the notation $\big\|\cdot\big\|_{H^s}$ for the usual norm of $H^s(\SS^\pm)$, and simply $\big\|\cdot\big\|_2$ for the $L^2(\SS^\pm)$ norm. We also for $s\in\RR$ and $k\in\mathbb{N}$ introduce the spaces
\[H^{s,k}(\SS^\pm)=\{f\in \mathcal D'(\overline{\SS^\pm}):\big\|f\big\|_{H^{s,k}}<\infty\},\]
where $\big\|f\big\|_{H^{s,k}}=\sum_{j=0}^{k}\big\|\Lambda^{s-j}\partial_z^j f\big\|_2$.

\subsection{The basic equations}
We assume that each fluid is irrotational and incompressible, so that we can introduce velocity potentials $\phi_i$ $(i=1,2)$ respectively associated to the upper and lower fluid layer. The velocity potentials satisfy 
\begin{equation}
 \Delta_{X,z} \phi_i=0 \mbox{ in }\Omega^i_t,
\label{laplace}
\end{equation}
where $\Omega^i_t$ denotes the domain of the fluid $i$ at time $t$ (see Figure \ref{Sketch}). Moreover, we assume the fluids to satisfy the Euler equation, and their respective density $\rho_i$ is constant, so that the velocity potentials satisfy the Bernoulli equation:
\begin{equation}
\label{Bernoulli}
 \partial_t \phi_i+\frac{1}{2} |\nabla_{X,z} \phi_i|^2=-\frac{P}{\rho_i}-gz \mbox{ in }\Omega^i_t,
\end{equation}
where $g$ denotes the acceleration of gravity and $P$ is the pressure inside the fluid. The kinematic boundary condition at the known, constant with respect to time, bottom topography ${\Gamma_b:=\{z=-h_{20}+b(X)\}}$ is given by
\begin{equation}
\label{boundarybottom}  \partial_{n}\phi_2 = 0 \mbox{ on } \Gamma_b.
\end{equation}
It is presumed that the surface and the interface are given as the graph of functions (respectively $\zeta_1(t,X)$ and $\zeta_2(t,X)$) which express the deviation from their rest position (respectively $(X,h_{10})$ and $(X,0)$) at the spatial coordinate $X$ and at time $t$. The assumption that no fluid particle crosses the surface or the interface gives the following kinematic boundary conditions:
\begin{equation}\label{boundaries} \begin{array}{rll}
  \partial_t \zeta_1 &=\sqrt{1+|\nabla\zeta_1|^2}\partial_{n}\phi_1  &\mbox{ on } \Gamma_1:=\{z=h_{10}+\zeta_1(t,X)\}, \\
  \partial_t \zeta_2 &=\sqrt{1+|\nabla\zeta_2|^2}\partial_{n}\phi_1 =\sqrt{1+|\nabla\zeta_2|^2}\partial_{n}\phi_2  &\text{ on } \Gamma_2:=\{z=\zeta_2(t,X)\}.
 \end{array}\end{equation}

Finally, we close the set of equations assuming that 
\begin{equation}
\label{pressure}
 P \mbox{ is constant at the surface, and continuous at the interface.}
\end{equation}
\begin{Rem}\label{RemSurfTension}
Unlike the water wave problem (air-water interface), the Cauchy problem associated with waves at the interface of two fluids of positive different densities is known to be ill-posed in Sobolev spaces in the absence of surface tension, as Kelvin-Helmholtz instabilities appear. However, when adding a small amount of surface tension, Lannes~\cite{Lannes6} proved that thanks to a stability criterion, the problem becomes well-posed with a time of existence that does not vanish as the surface tension goes to zero, and thus is consistent with the observations. The Kelvin-Helmholtz instabilities appear for high frequencies, where the regularization effect of the surface tension is relevant, while the main profile of the wave that we want to capture is located in lower frequencies, and is unaffected by surface tension. Therefore adding a small amount of surface tension at the interface\footnote{The study of Lannes focus on the two-layer fluid system with a rigid lid. However, we believe that the theory in the free surface case does not differ much from the one in the rigid lid configuration.} in the Euler system guarantees the well-posedness of the system and does not change our asymptotic models. For simplicity, we decide to omit this surface tension term.
\end{Rem}

\subsection{Reduction of the equations}
In~\cite{Zakharov}, Zakharov remarked that the surface wave system can be fully deduced from the knowledge of the surface elevation, and the trace of the velocity potential at the surface. We extend it here for two fluids in the free-surface case. Indeed, if we introduce the traces
\[\psi_1(t,X):=\phi_1(t,X,h_{10}+\zeta_1(t,X)) \mbox{, and } \psi_2(t,X):=\phi_2(t,X,\zeta_2(t,X)),\]
then $\phi_2$ is uniquely given as the solution of Laplace's equation~\eqref{laplace} in the lower fluid domain, with the Neumann condition~\eqref{boundarybottom} on $\Gamma_b$ and the Dirichlet condition $\phi_2=\psi_2$ on $\Gamma_2$. Then, $\phi_1$ is obtained as the solution of Laplace's equation on the upper fluid domain, with the Neumann condition given by~\eqref{boundaries} $\partial_{n}\phi_2=\partial_{n}\phi_1$ on $\Gamma_2$, and the Dirichlet condition $\phi_1=\psi_1$ on $\Gamma_1$.

Following the formalism introduced by Craig and Sulem in~\cite{CSS}, we first define the Dirichlet-Neumann operators:
\begin{eqnarray*}
 G_1[\zeta_1,\zeta_2,b](\psi_1,\psi_2)&:=&\sqrt{1+|\nabla\zeta_1|^2}\partial_{n}\phi_1{}_{|z=h_{10}+\zeta_1}, \\
 G_2[\zeta_2,b]\psi_2&:=&\sqrt{1+|\nabla\zeta_2|^2}\partial_{n}\phi_2{}_{|z=\zeta_2} .
\end{eqnarray*}
We also define the following operator:
\begin{equation*}
 H[\zeta_1,\zeta_2,b](\psi_1,\psi_2):=\nabla\phi_1{}_{|z=\zeta_2}.
\end{equation*}

Using the chain rule and the last definitions in the relation~\eqref{Bernoulli} evaluated at the surface, we obtain
\begin{equation}
\label{Bernsurf}
 \partial_t \psi_1+g(h_{10}+\zeta_1)+\frac{1}{2} |\nabla \psi_1|^2-\frac{(G_1[\zeta_1,\zeta_2,b](\psi_1,\psi_2)+\nabla\zeta_1\cdot\nabla\psi_1)^2}{2(1+|\nabla\zeta_1|^2)}=-\frac{P_1}{\rho_1},
\end{equation}
where $P_1$ is the constant pressure at the surface. Using again the Bernoulli equation for the upper and the lower fluid evaluated at the interface, we have
\begin{eqnarray}
\partial_t  ({\phi_1}_{|z=\zeta_2})+g\zeta_2+\frac{1}{2} | H[\zeta_1,\zeta_2,b](\psi_1,\psi_2)|^2&& \nonumber \\ \label{Berninterf1} -\frac{(G_2[\zeta_2,b]\psi_2+\nabla\zeta_2\cdot H[\zeta_1,\zeta_2,b](\psi_1,\psi_2))^2}{2(1+|\nabla\zeta_2|^2)}&=&-\frac{P_2}{\rho_1}, \\
\label{Berninterf2}\partial_t \psi_2+g\zeta_2+\frac{1}{2} |\nabla \psi_2|^2-\frac{(G_2[\zeta_2,b]\psi_2+\nabla\zeta_2\cdot\nabla\psi_2)^2}{2(1+|\nabla\zeta_2|^2)}&=&-\frac{P_2}{\rho_2},
\end{eqnarray}
where $P_2$ is the pressure at the interface, identical in~\eqref{Berninterf1} and~\eqref{Berninterf2}, thanks to the continuity assumption in~\eqref{pressure}.

Finally, using~\eqref{boundaries}, the gradient of the equality~\eqref{Bernsurf} and a straightforward combination of~\eqref{Berninterf1} and~\eqref{Berninterf2}, we obtain the system of equations
\begin{equation}
\label{dimsys}
\left\{ \begin{array}{lr}
 \multicolumn{2}{l}{\partial_t\zeta_1-G_1[\zeta_1,\zeta_2,b](\psi_1,\psi_2)=0,}  \\
\multicolumn{2}{l}{\partial_t\zeta_2-G_2[\zeta_2,b]\psi_2=0,}  \\
\multicolumn{2}{l}{\partial_t \nabla\psi_1+g\nabla\zeta_1+\frac{1}{2} \nabla(|\nabla \psi_1|^2)-\nabla\NN_1(\zeta_1,\zeta_2,b,\psi_1,\psi_2)=0,} \\
\multicolumn{2}{l}{\partial_t (\nabla\psi_2-\gamma H[\zeta_1,\zeta_2,b](\psi_1,\psi_2))+g(1-\gamma)\nabla\zeta_2} \\
&+\frac{1}{2} \nabla(|\nabla \psi_2|^2-\gamma |H[\zeta_1,\zeta_2,b](\psi_1,\psi_2)|^2) -\nabla\NN_2(\zeta_1,\zeta_2,b,\psi_1,\psi_2)=0,
\end{array} 
\right. 
\end{equation}
where $\gamma=\frac{\rho_1}{\rho_2}$, and \[\begin{array}{ll} \NN_1(\zeta_1,\zeta_2,b,\psi_1,\psi_2)  &= \frac{(G_1[\zeta_1,\zeta_2,b](\psi_1,\psi_2)+\nabla\zeta_1\cdot\nabla\psi_1)^2}{2(1+|\nabla\zeta_1|^2)},\\
 \NN_2(\zeta_1,\zeta_2,b,\psi_1,\psi_2)  &= \frac{(G_2[\zeta_2,b]\psi_2+\nabla\zeta_2\cdot\nabla\psi_2)^2-\gamma(G_2[\zeta_2,b]\psi_2+\nabla\zeta_2\cdot H[\zeta_1,\zeta_2,b](\psi_1,\psi_2))^2}{2(1+|\nabla\zeta_2|^2)}.\end{array}\]
This is the system of equations that we use in order to derive asymptotic models.

\subsection{Nondimensionalization of the equations}
In this subsection, we rewrite the system~\eqref{dimsys} in dimensionless variables, introducing dimensionless parameters which are crucial to study the asymptotic dynamics. We denote by $a_1$ the typical amplitude of the surface deformation, and by $a_2$ that of the interface. $\lambda$ is a characteristic horizontal length, say the wavelength of the interface. Finally, $B$ is the order of bottom topography variation.

We define the dimensionless variables 
\[\begin{array}{cccc}
 \t X:=\dfrac{X}{\lambda}, & \t z:=\dfrac{z}{h_{10}}, & \t t:=\dfrac{t}{\lambda/\sqrt{gh_{10}}}, & \t b(\t X):=\dfrac{b(X)}{B},
\end{array}\]
and the dimensionless unknowns
\[\begin{array}{cc}
 \t{\zeta_i}(\t X):=\dfrac{\zeta_i(X)}{a_i}, & \t{\psi_i}(\t X):=\dfrac{\psi_i(X)}{a_2\lambda\sqrt{g/h_{10}}}.
\end{array}\]
Five independent parameters of the system are thus added to $\gamma=\frac{\rho_1}{\rho_2}$:
\[\begin{array}{ccccc}
 \epsilon_1:=\dfrac{a_1}{h_{10}}, & \epsilon_2:=\dfrac{a_2}{h_{10}}, & \mu:=\dfrac{h_{10}^2}{\lambda^2}, & \delta:=\dfrac{h_{10}}{h_{20}}, & \beta:=\dfrac{B}{h_{10}}.
\end{array}\]
So, $\epsilon_1$ and $\epsilon_2$ are the nonlinearity parameters and $\mu$ is the shallowness parameter. We also define the convenient notation
\[\alpha:=\dfrac{a_1}{a_2}=\dfrac{\epsilon_1}{\epsilon_2}.\]

\begin{Rem}
 The scaling for nondimensionalization has been chosen considering the solutions of the linearized problem, that can be computed with the physical variables using the method of \S~\ref{linearized} (see~\cite{SeH} for example). Using such a scaling, we implicitly assume that the two layers are of similar depth (i.e. $\delta\sim 1$). Therefore, the choice of $h_{10}$ (and not $h_{20}$) as the reference vertical length and $\sqrt{gh_{10}}$ as the reference velocity are harmless. We refer for example to ~\cite{Kub,KoB} for the investigation of different situations such as the deep-water regime, or the finite-depth regime.
 
 In the same way, we decide to use the same scaling on $\psi_1$ and $\psi_2$ in order to simplify the Definitions~\ref{scale1} and~\ref{scale2}, and especially keep the relation $\partial_{n}\phi_1 =\partial_{n}\phi_2$ on the interface. We choose $a_2$ instead of $a_1$, so that the expansions of \S~\ref{opexpansion} hold for $\alpha$ tending to zero. In that way, we are able to retrieve the Shallow water/Shallow water with rigid-lid model, in \S~\ref{rigidlid}.
 
 Finally, as we choose a unique characteristic horizontal length, we only focus on the case where the internal and the surface waves have length scale of the same order, and hence do not consider phenomenons such as the resonant interaction between a long internal wave and short surface waves, as studied for example in~\cite{FuO}. Moreover, in the case of three-dimensional waves ($d=2$), a unique characteristic horizontal length means that there is no preferential horizontal direction, so that we do not study transverse waves.
\end{Rem}

We now rewrite the system in terms of dimensionless variables. First, we have to define the dimensionless operators, associated to the dimensionless fluid domains:
\begin{eqnarray*}
 \Omega_1&:=&\{(X,z)\in \RR^{d+1}, \epsilon_2{\zeta_2}(X)<z<1+\epsilon_1\zeta_1(X)\},\\
 \Omega_2&:=&\{(X,z)\in \RR^{d+1}, -\frac{1}{\delta}+\beta b(X)<z<\epsilon_2\zeta_2(X)\}.
\end{eqnarray*}
In the following, we always assume that the domains remain strictly connected, so there is a positive value $h$ such that for all $X\in\RR^d$,
\begin{equation}
\label{h}
 1+\epsilon_1\zeta_1(X)-\epsilon_2{\zeta_2}(X)\geq h>0 \qquad \mbox{and} \qquad \frac{1}{\delta}+\epsilon_2\zeta_2(X)-\beta b(X)\geq h>0.
\end{equation}

\begin{definition}\label{scale1}
 Let $\zeta_2$ and $b \in W^{1,\infty}(\RR^d)$, such that $\Omega_2$ satisfies~\eqref{h}, and suppose that $\nabla\psi_2 \in H^{1/2}(\RR^d)$. Then with $\phi_2$ the unique solution in $H^{2}(\Omega_2)$ of the boundary value problem
\begin{equation}\label{phi2}\left\{\begin{array}{ll}
 \Delta^\mu_{X,z}\phi_2=0 & \mbox{in } \Omega_2, \\
\phi_2 =\psi_2 & \mbox{on } \Gamma_2:=\{z=\epsilon_2 \zeta_2\}, \\
 \partial_{n}\phi_2 =0 & \mbox{on } \Gamma_b:=\{z=-\frac{1}{\delta}+\beta b\}, \\

\end{array}
\right.\end{equation}
we define $G_2^{\mu,\delta}[\epsilon_2\zeta_2,\beta b]\psi_2 \in H^{1/2}(\RR^d)$ by
\[G_2^{\mu,\delta}[\epsilon_2\zeta_2, \beta b]\psi_2:=-\mu\epsilon_2\nabla\zeta_2\cdot \nabla\phi_2{}_{|z=\epsilon_2\zeta_2}+\partial_z\phi_2{}_{|z=\epsilon_2\zeta_2}.\]
\end{definition}
\begin{definition}\label{scale2}
 Let now $\zeta_1$, $\zeta_2$, and $b \in W^{1,\infty}(\RR^d)$ be such that $\Omega_1$ and $\Omega_2$ satisfy~\eqref{h}, and suppose $\nabla\psi_1$, $\nabla\psi_2 \in H^{1/2}(\RR^d)$. Let $\phi_1$ be the unique solution in $H^{2}(\Omega_2)$ of the boundary value problem
\begin{equation}\label{phi1}\left\{\begin{array}{ll}
 \Delta^\mu_{X,z}\phi_1=0 & \mbox{in } \Omega_1, \\
\phi_1 =\psi_1 & \mbox{on } \Gamma_1:=\{z=1+\epsilon_1 \zeta_1\}, \\
 \partial_{n}\phi_1 =\frac{1}{\sqrt{1+\mu\epsilon_2^2|\nabla \zeta_2|^2}}G_2^{\mu,\delta}[\epsilon_2\zeta_2, \beta b]\psi_2 & \mbox{on } \Gamma_2.
\end{array}
\right.\end{equation}
Then we define $G_1^{\mu,\delta}[\epsilon_1\zeta_1,\epsilon_2\zeta_2,\beta b](\psi_1,\psi_2) \in H^{1/2}(\RR^d)$ by
\[G_1^{\mu,\delta}[\epsilon_1\zeta_1,\epsilon_2\zeta_2,\beta b](\psi_1,\psi_2):=-\mu\epsilon_1\nabla\zeta_1\cdot \nabla\phi_1{}_{|z=1+\epsilon_1\zeta_1}+\partial_z\phi_1{}_{|z=1+\epsilon_1\zeta_1},\]
and $H^{\mu,\delta}[\epsilon_1\zeta_1,\epsilon_2\zeta_2,\beta b](\psi_1,\psi_2) \in H^{1/2}(\RR^d)$ by
\[H^{\mu,\delta}[\epsilon_1\zeta_1,\epsilon_2\zeta_2, \beta b](\psi_1,\psi_2) =(\nabla \phi_1)_{|z=\epsilon_2\zeta_2} .\]
\end{definition}
In the following, when there is no possibility of mistake, we simply write:
\begin{eqnarray*}
 G_2\psi_2&:=&G_2^{\mu,\delta}[\epsilon_2\zeta_2, \beta b]\psi_2, \\
 G_1(\psi_1,\psi_2)&:=&G_1^{\mu,\delta}[\epsilon_1\zeta_1,\epsilon_2\zeta_2,\beta b](\psi_1,\psi_2), \\
 H(\psi_1,\psi_2)&:=& H^{\mu,\delta}[\epsilon_1\zeta_1,\epsilon_2\zeta_2, \beta b](\psi_1,\psi_2).
\end{eqnarray*}
\begin{Rem}
 The existence and uniqueness of such solutions $\phi_2$ and $\phi_1$ are given by Proposition~\ref{flattening}.
\end{Rem}

Using these last definitions, it is straightforward to check that the system~\eqref{dimsys} becomes in dimensionless variables (where we omit the tildes for the sake of clarity):
\begin{equation}
\label{sys}
\left\{ \begin{array}{l}
\alpha\partial_{ t}{\zeta_1}-\frac{1}{\mu}G_1(\psi_1,\psi_2)=0,  \\
\partial_{ t}{\zeta_2}-\frac{1}{\mu}G_2\psi_2=0,  \\
\partial_{ t} \nabla{\psi_1}+\alpha\nabla{\zeta_1}+\frac{\epsilon_2}{2} \nabla(|\nabla {\psi_1}|^2)=\mu\epsilon_2\nabla\NN_1, \\
\partial_{ t} (\nabla{\psi_2}-\gamma H(\psi_1,\psi_2))+(1-\gamma)\nabla{\zeta_2} +\frac{\epsilon_2}{2} \nabla(|\nabla {\psi_2}|^2-\gamma |H(\psi_1,\psi_2)|^2) =\mu\epsilon_2\nabla\NN_2,
\end{array} 
\right. 
\end{equation}
where \begin{eqnarray*}\NN_1  &:=& \dfrac{(\frac{1}{\mu}G_1(\psi_1,\psi_2)+\epsilon_1\nabla{\zeta_1}\cdot\nabla{\psi_1})^2}{2(1+\mu|\epsilon_1\nabla{\zeta_1}|^2)},\\
 \NN_2 &:=& \dfrac{(\frac{1}{\mu}G_2\psi_2+\epsilon_2\nabla{\zeta_2}\cdot\nabla{\psi_2})^2-\gamma(\frac{1}{\mu}G_2\psi_2+\epsilon_2\nabla{\zeta_2}\cdot H(\psi_1,\psi_2))^2}{2(1+\mu|\epsilon_2\nabla{\zeta_2}|^2)}.       
      \end{eqnarray*}
We derive the asymptotic models from this system of non-dimensionalized equations, corresponding to different sizes for the dimensionless parameters.

\subsection{The linearized equation}
\label{linearized}
Linearizing the system~\eqref{sys} around the rest state, we obtain
\begin{equation}
\label{linsys}
\left\{ \begin{array}{l}
\alpha\partial_t \zeta_1-\frac{1}{\mu}G_1^{\mu,\delta}[0,0,0](\psi_1,\psi_2)=0,  \\
\partial_t\zeta_2-\frac{1}{\mu}G_2^{\mu,\delta}[0,0]\psi_2=0,  \\
\partial_t \nabla\psi_1+\alpha\nabla\zeta_1=0, \\
\partial_t (\nabla\psi_2-\gamma H^{\mu,\delta}[0,0,0](\psi_1,\psi_2))+(1-\gamma)\nabla\zeta_2 =0.
\end{array} 
\right. 
\end{equation}
Now, when the surface, the interface and the bottom are flat, we have explicit expressions for the operators $G_1$, $G_2$ and $H$. Indeed, taking the horizontal Fourier transform of the Laplace equations in~\eqref{phi2} and~\eqref{phi1}, we obtain that $\widehat{\phi_2}$ and $\widehat{\phi_1}$ are solutions of the following ordinary differential equations:
\[-\mu|D|^2y+y''=0.\]
Then, using the boundary conditions, we deduce
\[
 \phi_2(X,z)= \cosh(\sqrt{\mu}|D|z)\psi_2(X)+\tanh(\frac{\sqrt{\mu}}{\delta}|D|)\sinh(\sqrt{\mu}|D|z)\psi_2(X),
\]
so that we have
\[G_2^{\mu,\delta}[0,0]\psi_2=\sqrt{\mu}|D|\tanh(\frac{\sqrt{\mu}}{\delta}|D|)\psi_2.\]
Then we obtain
\begin{eqnarray*}
  \phi_1(X,z)=&&\cosh(\sqrt{\mu}|D|z)\big(\frac{1}{\cosh(\sqrt{\mu}|D|)}\psi_1(X)-\tanh(\frac{\sqrt{\mu}}{\delta}|D|)\tanh(\sqrt{\mu}|D|)\psi_2(X) \big)\\
&&\quad +\tanh(\frac{\sqrt{\mu}}{\delta}|D|)\sinh(\sqrt{\mu}|D|z)\psi_2(X),
\end{eqnarray*}
so that we have
\[G_1^{\mu,\delta}[0,0,0](\psi_1,\psi_2)=\frac{\sqrt{\mu}|D|}{\cosh(\sqrt{\mu}|D|)}\left(\sinh(\sqrt{\mu}|D|)\psi_1+\tanh(\frac{\sqrt{\mu}}{\delta})\psi_2\right),\]
and finally
\[H^{\mu,\delta}[0,0,0](\psi_1,\psi_2)=\frac{1}{\cosh(\sqrt{\mu}|D|)}\nabla\psi_1(X)-\tanh(\frac{\sqrt{\mu}}{\delta}|D|)\tanh(\sqrt{\mu}|D|)\nabla\psi_2(X).\]

Using these expressions in the system~\eqref{linsys}, we can easily calculate the dispersion relations. Indeed, the wave frequencies $\omega_\pm^2(k)$, corresponding to plane-wave solutions $e^{ik\cdot X-i\omega_\pm t}$, are the solutions of the quadratic equation
\begin{equation}
\label{lineq} \omega^4-\frac{|k|}{\sqrt{\mu}}\dfrac{\tanh(\sqrt{\mu}|k|)+\tanh(\frac{\sqrt{\mu}}{\delta}|k|)}{1+\gamma\tanh(\sqrt{\mu}|k|)\tanh(\frac{\sqrt{\mu}}{\delta}|k|)}\omega^2+\frac{|k|^2}{\mu}\dfrac{(1-\gamma)\tanh(\sqrt{\mu}|k|)\tanh(\frac{\sqrt{\mu}}{\delta}|k|)}{1+\gamma\tanh(\sqrt{\mu}|k|)\tanh(\frac{\sqrt{\mu}}{\delta}|k|)}=0.
\end{equation}
This equation has two strictly positive solutions (and their opposite) if and only if $\gamma<1$, corresponding to the case wherein the lower fluid is heavier than the upper one. This expression also appears in~\cite{CGK} and~\cite{PeS}. The figure~\ref{dispfullsys} shows the evolution of the wave frequencies $\omega_\pm$, $-\omega_\pm$, as functions of the wave number $k$. We chose the parameters $\mu=0.1$, $\delta=1/3$, $\gamma=2/3$.

\begin{figure}[htb]
\centering
\psfrag{k}{\begin{footnotesize}$k$                          \end{footnotesize}}
\psfrag{w}{\begin{footnotesize}$\omega$                          \end{footnotesize}}
\psfrag{wmmm}{\begin{footnotesize}$\pm\omega_-$                          \end{footnotesize}}
\psfrag{wppp}{\begin{footnotesize}$\pm\omega_+$                          \end{footnotesize}}
 \includegraphics[width=0.7\textwidth]{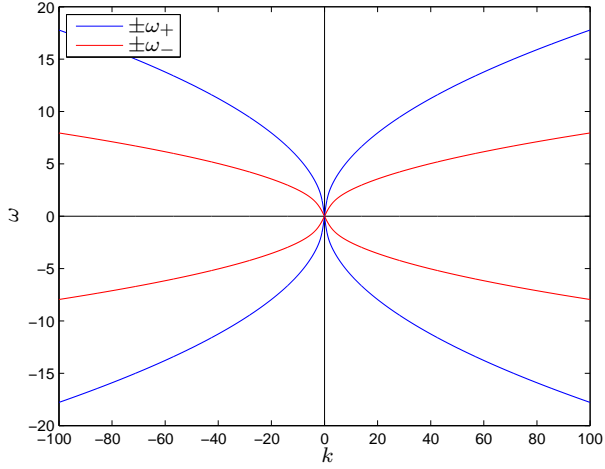}
 \caption{Full system dispersion}
\label{dispfullsys}
\end{figure}

\begin{Rem}
 We remark that setting $\gamma = 0$, and $\delta = 1$, one recovers the classical dispersion relation for the one-fluid system:
 \[\omega^2=\frac{|k|}{\sqrt{\mu}}\tanh(\sqrt{\mu}|k|).\]
\end{Rem}

\section{Asymptotic Models}
\label{asymptexp}
We derive asymptotic models for the system, by obtaining explicit expansions of the operators. Following the method of~\cite{BLS}, it is convenient to first reduce the problems~\eqref{phi2} and~\eqref{phi1} to elliptic equations on a flat strip.
\subsection{Flattening of the domain}
We define the mappings
\begin{eqnarray*}
 R_1&:= \begin{array}{ccc}
 \Omega_1 & \rightarrow & \SS^+ \\
 (X,z) & \mapsto & (X,r_1(X,z))
\end{array} &\mbox{ with } r_1(X,z):=\dfrac{z-\epsilon_2\zeta_2(X)}{1+\epsilon_1\zeta_1(X)-\epsilon_2\zeta_2(X)}, \\
 R_2&:=\begin{array}{ccc}
  \Omega_2 & \rightarrow & \SS^- \\
 (X,z) & \mapsto & (X,r_2(X,z))
\end{array} &\mbox{ with } r_2(X,z):=\dfrac{z-\epsilon_2\zeta_2(X)}{1/\delta-\beta b(X)+\epsilon_2\zeta_2(X)},
\end{eqnarray*}
and denote their inverse
\[
 S_1:= \begin{array}{ccc}
 \SS^+ & \rightarrow & \Omega_1 \\
 (  X,  z) & \mapsto & (  X,s_1(  X,  z))
\end{array} \ \mbox{ and } \ \
 S_2:=\begin{array}{ccc}
\SS^- & \rightarrow & \Omega_2 \\
 (  X,  z) & \mapsto & (  X,s_2(  X,  z))
\end{array}  ,
\]
with
\begin{eqnarray*}
 s_1(  X,  z)&:=&(1+\epsilon_1\zeta_1(  X)-\epsilon_2\zeta_2(  X))  z +\epsilon_2\zeta_2(  X), \\
 s_2(  X,  z)&:=&(1/\delta-\beta b(  X)+\epsilon_2\zeta_2(  X))  z +\epsilon_2\zeta_2(  X).
\end{eqnarray*}
Introducing the $(d+1)\times(d+1)$ matrices
\begin{eqnarray}\label{Pi} {P_i}&:=&\frac{1}{\partial_{  z}s_i}
\left(\begin{array}{cc}  \partial_{  z}s_i I_d & 0_{d,1} \\ -{\nabla_{  X}s_i}^T & 1 \end{array}\right)
\left(\begin{array}{cc}  \mu I_d & 0_{d,1} \\ 0_{1,d} & 1 \end{array}\right)
\left(\begin{array}{cc}  \partial_{  z}s_i I_d & -\nabla_{  X}s_i \\ 0_{1,d} & 1 \end{array}\right)\nonumber \\
&=&\left(\begin{array}{cc}  \mu \partial_{  z}s_i I_d & -\mu\nabla_{  X}s_i \\ -\mu{\nabla_{  X}s_i}^T & \frac{1+\mu|\nabla_{  X}s_i|^2}{\partial_{  z}s_i} \end{array}\right),\end{eqnarray}
where $0_{m,n}$ is the $m\times n$ zero matrix and $I_d$ the $d\times d$ identity matrix, we can transform the Laplace equations~\eqref{phi2} and~\eqref{phi1} into elliptic boundary value problems on flat strips. 
\begin{proposition} \label{flattening}
 Let $\zeta_1$, $\zeta_2$, and $b \in W^{1,\infty}(\RR^d)$, such that $\Omega_1$ and $\Omega_2$ satisfy~\eqref{h}, and suppose $\nabla\psi_1$, $\nabla\psi_2 \in H^{1/2}(\RR^d)$. Then there exists a unique solution $\phi_1\in H^2(\SS^+)$ and $\phi_2\in H^2(\SS^-)$ to the following boundary value problems
\begin{equation}\label{tphi2}\left\{\begin{array}{ll}
 \nabla_{  X,  z}\cdot P_2 \nabla_{  X,  z} {\phi_2}=0 & \mbox{in } \SS^-, \\
 {\phi_2} =\psi_2 & \mbox{on } \{  z=0\}, \\
 \partial_n {\phi_2}=0 & \mbox{on } \{z=-1\}, \\

\end{array}
\right.\end{equation}
and
\begin{equation}\label{tphi1}\left\{\begin{array}{ll}
 \nabla_{  X,  z}\cdot P_1 \nabla_{  X,  z} {\phi_1}=0 & \mbox{in } \SS^+, \\
{\phi_1} =\psi_1 & \mbox{on } \{  z=1\}, \\
 \partial_n {\phi_1} =\partial_n{\phi_2} & \mbox{on } \{  z=0\},
\end{array}
\right.\end{equation}
where $\partial_n\phi$ stands for the upward co-normal derivative associated to the elliptic operator involved: \[\partial_n\phi:=e_{d+1}\cdot P\nabla_{X,z}\phi.\]
Moreover, $\t{\phi_i}:=(  X,  z)\mapsto \phi_i(  X,r_i(  X,  z))$ $(i=1,2)$ respectively solve the problems~\eqref{phi1} and~\eqref{phi2}. Thus, the operators $G_1$, $G_2$ and $H$ can equivalently be defined with 
\begin{eqnarray*}
G_2 {\psi_2}&=&e_{d+1}\cdot   P_2\nabla_{  X,  z} {\phi_2} {}_{|  z=0}, \\
G_1( \psi_1, \psi_2)&=&e_{d+1}\cdot P_1\nabla_{  X,  z} {\phi_1} {}_{|  z=1}, \\
H( \psi_1, \psi_2)&=&\nabla  {\phi_1} {}_{|  z=0}.
\end{eqnarray*}
\end{proposition}
\begin{proof}
The reduction of the problems~\eqref{phi1} and~\eqref{phi2} on the flat strip can be found on~\cite{Lannes1} (Proposition 2.7).
The coercivity condition is satisfied thanks to~\eqref{h} and the assumptions on $\zeta_1$, $\zeta_2$ (see Proposition 2.3 of~\cite{Lannes2}): 
\begin{equation}\label{asP}\exists k>0, \forall \Theta \in \RR^{d+1},
\Theta\cdot P_i\Theta \geq \frac{1}{k} \big\vert \Theta\big\vert ^2.\end{equation}

Thus, we just prove here the existence and uniqueness of the $H^2$-solutions $\phi_i$ $(i=1,2)$.
Since for all $g\in H^{-1/2}(\RR^d),h\in H^{1/2}(\RR^d)$, one can easily construct a function $w\in H^1(\SS^+)$ such that $w_{|z=1}=h$ and $\partial_n w_{|z=0}=g$,~\eqref{tphi1} and~\eqref{tphi2} clearly reduce to the following problem
\begin{equation}\label{elliptic}\left\{\begin{array}{ll}
 \nabla_{X,z}\cdot P \nabla_{  X,  z} {\phi_1}=f & \mbox{in } \SS^+, \\
{\phi_1} =0 & \mbox{on } \Gamma_1:=\RR^d\times\{1\}, \\
 \partial_n {\phi_1} =0 & \mbox{on } \Gamma_2:=\RR^d\times\{0\},
\end{array}
\right.\end{equation}
where $f\in H^{-1}(\SS^+)$ and $P$ satisfies~\eqref{asP}.

As a first step, we introduce the variational formulation of this problem. Let us define the functional space
\[V:=\{v\in H^1(\SS^+), \gamma_0(v)=0 \mbox{ on } \RR^d\},\]
with $\gamma_0:H^1(\SS^+) \rightarrow H^{1/2}(\RR^d)$ the trace operator on $\Gamma_1$. Since $\gamma_0$ is continuous, $V$, equipped with the scalar product of $H^1(\SS^+)$ and the corresponding norm, is a closed subspace of $H^1(\SS^+)$, hence a Hilbert space. A solution of the variational problem related to~\eqref{elliptic} is then a function $u \in V$ such that
\begin{equation} \label{var}
 \forall v\in V, \int_{\SS^+}P\nabla u\cdot\nabla v = -\int_{\SS^+} fv.
\end{equation}
Since $\mathcal V =\{v\in \mathcal D(\bar\SS^+),v=0 \mbox{ on } \Gamma_1\}$ is dense in $V$, a solution of the variational problem~\eqref{var} is a weak solution of the problem~\eqref{elliptic}.

Now we can check that $a(u,v):=\int_{\SS^+}P\nabla u\cdot\nabla v$ is a continuous bilinear form. The coercivity of $a$ is given by~\eqref{asP} and a generalized Poincar\'e inequality (see~\cite{ABM}, Theorem 5.4.3). Finally, since $b:v\in V \mapsto -\int_{\SS^+} fv$ is clearly continuous, the Lax-Milgram Theorem gives the existence and uniqueness of a solution $u\in V$ of~\eqref{var}, and thus a weak solution of~\eqref{elliptic}. Moreover, one has
\[\big\|u\big\|_{H^1} \leq C \big\|f\big\|_{H^{-1}}.\]

The last step consists in proving that the solution $u$ lives in $H^2(\SS^+)$, if we assume that $f\in L^2$. We introduce for $h>0$, 
\[u_h:=(x,y,z)\mapsto \frac{\tau_h u(x,y,z)-u(x,y,z)}{h}=\frac{u(x+h,y,z)-u(x,y,z)}{h}.\]
Then $u_h$ is the solution~\eqref{elliptic} with $f_h=\frac{\tau_h f-f}{h}$ and $g_h=\frac{\tau_h g-g}{h}$, so that
\[\big\|u_h\big\|_{H^1} \leq C\big\|f_h\big\|_{H^{-1}}.\]
Then we remark that for any $v\in H^{1}(\SS^+)$, $v_h(x,y,z)=\frac{1}{h}\int_x^{x+h} \partial_x v(t,y,z) dt$, so that
\[\big\|v_h\big\|_{L^2}\leq  \frac{1}{h}\int_0^{h} \big\|\partial_x v\big\|_{L^2} dt\leq\big\| v\big\|_{H^1}.\]
Thus, one has thanks to the duality between $H^1$ and $H_0^1$,
\[ \big\|f_h\big\|_{H^{-1}} \leq \sup_{v\in H_0^1} \frac{|(f_h,v)|}{\big\| v\big\|_{H^1}} 
\leq \sup_{v\in H_0^1} \frac{\big\|f\big\|_{L^2}\big\|v_h\big\|_{L^2}}{\big\| v\big\|_{H^1}} \leq \big\|f\big\|_{L^2}.\]
We finally have 
\[\big\|u_h\big\|_{H^1} \leq C \big\|f\big\|_{L^2}.\]
Since $V$ is a Hilbert space, we deduce that there exists $w \in V$ and a subsequence $(u_{{h_k}})$ such that $u_{{h_k}}$ weakly converges towards $w$. Moreover, we know that $u_{{h_k}}$ converges towards $\partial_x u$ in $\mathcal D'(\SS^+)$, so we deduce $\partial_x u\in V\subset H^1$.

We prove in the same way that $\partial_y u\in H^1$, so that $\Delta_X u \in L^2$. Finally, thanks to~\eqref{asP}, we have
\[|\partial_z^2 u|\leq |\Delta_X u|+k|\nabla_{  X,  z}\cdot P \nabla_{  X,  z} u|=|\Delta_X u|+k|f|,\]
so that $u\in H^2(\SS^+)$, and the Proposition is proved.
\end{proof}

\subsection{Asymptotic expansion of the operators}
\label{opexpansion}
We are looking for shallow-water models ($\mu\ll 1$), and therefore need to obtain an expansion of the operators in terms of $\mu$. The method  is the following. We first exhibit the expansion of the matrix $P_i$ in terms of $\mu$. Then we look for approximate solutions $\phi_i^{app}$ ($i=1,2$) under the form:\[\phi_i^{app}=\phi_i^0+\mu \phi_i^1+ \mu^2\phi_i^2.\]
Plugging this Ansatz into~\eqref{tphi2} and~\eqref{tphi1}, and solving at each order of $\mu$, gives the $\phi_i^j$. From which we can deduce the expansion of the operators, by computing the normal derivative of $\phi_i^{app}$.

Since~\eqref{tphi2} is exactly the same problem as involved (in the case of the water-wave) in~\cite{Lannes2}, we can directly apply the Proposition 3.8 to the lower fluid.
\begin{proposition}
\label{defT}
Let $t_0>d/2$ and $s\geq t_0+1/2$, $\nabla \psi_2 \in H^{s+11/2}(\RR^2)$, ${\zeta_2 \in H^{s+9/2}(\RR^2)}$ and $b \in H^{s+11/2}(\RR^2)$, such that~\eqref{h} is satisfied. Then one has
 \begin{eqnarray}
\label{G20}  \big\vert G_2\psi_2+\mu\nabla\cdot  (h_2\nabla\psi_2)\big\vert_{H^s}&\leq& \mu^2 C_0,\\
\label{G2}  \big\vert G_2\psi_2+\mu\nabla\cdot  (h_2\nabla\psi_2)-\mu^2\nabla \cdot \mathcal T[h_2,\beta b]\nabla\psi_2\big\vert_{H^s}&\leq& \mu^3 C_1,
 \end{eqnarray}
with $C_j=C(\frac{1}{h},\beta\big\vert b\big\vert_{H^{s+7/2+2j}},\epsilon_2\big\vert\zeta_2\big\vert_{H^{s+5/2+2j}},\big\vert\nabla\psi_2\big\vert_{H^{s+7/2+2j}})$, and where we denote by ${h_2:=\frac{1}{\delta}-\beta b + \epsilon_2\zeta_2}$ the thickness of the lower layer, and
\[\mathcal T[h,b]V:=-\frac{1}{3}\nabla(h^3\nabla\cdot V)+\frac{1}{2}\big(\nabla(h^2\nabla b \cdot V)-h^2\nabla b\nabla \cdot V\big) +h\nabla b\nabla b \cdot V.\]
\end{proposition}
\begin{Rem} \label{remphi}
To obtain the estimate~\eqref{G20}, we use the approximate solution
\[\phi_2^{app,1}=\psi_2-\mu h_2\big(h_2(\frac{z^2}{2}+z)\Delta\psi_2-z\beta\nabla b\cdot\nabla\psi_2\big).\]
We need a higher order approximation to obtain~\eqref{G2}, namely $\phi_2^{app,2}=\phi_2^{app,1}+\mu^2 \phi_2^2$, where $\phi_2^2$ can be obtained using the same method as in the following study. The Proposition~\ref{defT} is then obtained following the path of Appendix~\ref{PreuveG} for the lower fluid (see~\cite{Florent} for a rigorous proof).
\end{Rem}

The study of the upper fluid is different from the one of the lower fluid, since we have now a non homogeneous Neumann condition on the interface. In order to manage this, we first decompose $\phi_1:=\check{\phi_1}+\bar{\phi_1}$, where $\check{\phi_1}$ is the unique solution of
\begin{equation}\label{cphi1}\left\{\begin{array}{ll}
 \nabla_{X,z}\cdot P_1 \nabla_{X,z}\check{\phi_1}=0 & \mbox{in } \SS^+, \\
\check{\phi_1} =\psi_1 & \mbox{on } \{z=1\}, \\
 \partial_n\check{\phi_1} =0 & \mbox{on } \{z=0\},
\end{array}
\right.\end{equation}
and $\bar{\phi_1}$ is the unique solution of
\begin{equation}\label{bphi1}\left\{\begin{array}{ll}
 \nabla_{X,z}\cdot P_1 \nabla_{X,z}\bar{\phi_1}=0 & \mbox{in } \SS^+, \\
\bar{\phi_1} =0 & \mbox{on } \{z=1\}, \\
 \partial_n\bar{\phi_1} =G_2\psi_2 & \mbox{on } \{z=0\}.
\end{array}
\right.\end{equation}
Again, the system satisfied by $\check{\phi_1}$ reduces to the water-wave problem (where the topography of the bottom would be given by $\epsilon_2\zeta_2$), so we introduce as in Remark~\ref{remphi} the approximate solutions
\begin{eqnarray*}\check{\phi}_1^{app,1}&:=&\psi_1-\mu h_1\big(h_1(\frac{(z-1)^2}{2}+(z-1))\Delta\psi_1-(z-1)\epsilon_2\nabla\zeta_2\cdot\nabla\psi_1\big),\\
\check{\phi}_1^{app,2}&:=&\check{\phi}_1^{app,1}+\mu^2 \check{\phi}_1 ^2.
\end{eqnarray*}
It follows that $\check{G_1}\psi_1$ the contribution on the Dirichlet-Neumann operator from $\check{\phi_1}$ can be expanded as in the following Proposition.
\begin{proposition}
\label{G1c}
Let $t_0>d/2$ and $s\geq t_0+1/2$, $\nabla \psi_1 , \zeta_2 \in H^{s+11/2}(\RR^2)$, $\zeta_1 \in H^{s+9/2}(\RR^2)$, such that~\eqref{h} is satisfied. Then one has
\begin{eqnarray}
\big\vert\check{G_1}\psi_1+\mu\nabla\cdot  (h_1\nabla\psi_1)\big\vert_{H^s}&\leq& \mu^2 C_0,\\
\big\vert\check{G_1}\psi_1+\mu\nabla\cdot  (h_1\nabla\psi_1)-\mu^2\nabla \cdot \mathcal T[h_1,\epsilon_2\zeta_2]\nabla\psi_1\big\vert_{H^s}&\leq& \mu^3 C_1,
 \end{eqnarray}
with $C_j=C(\frac{1}{h},\epsilon_2\big\vert\zeta_2\big\vert_{H^{s+7/2+2j}},\epsilon_1\big\vert\zeta_1\big\vert_{H^{s+5/2+2j}},\big\vert\nabla\psi_1\big\vert_{H^{s+7/2+2j}})$, and where we denote by $h_1:=1+\epsilon_1\zeta_1-\epsilon_2\zeta_2$ the thickness of the upper layer, and $\mathcal T[h,b]V$ is defined as in Proposition~\ref{defT}.
\end{proposition}

The last step consists in computing the contribution on the Dirichlet-Neumann operator from $\bar{\phi_1}$. We first define $\bar{\phi}_1^{app}=\phi^0+\mu \phi^1+ \mu^2\phi^2$. It is straightforward that 
\[P_1=P^0+\mu P^1 \mbox{, with } P^0:=\left(\begin{array}{cc}  0_{d,d} & 0_{d,1} \\ 0_{1,d} & \frac{1}{h_1} \end{array}\right) \mbox{ and } P^1:=\left(\begin{array}{cc}  h_1 I_d & -\nabla_X s_1 \\ -{\nabla_X s_1}^T & \frac{|\nabla_X s_1|^2}{h_1} \end{array}\right), \]
where we have used the notations $0_{m,n}$ for the $m\times n$ zero matrix, and $I_d$ for the $d\times d$ identity matrix.
Plugging these expansions into~\eqref{bphi1}, using Proposition~\ref{defT}, and solving at each order, we get:

\paragraph{At order $O(1)$}
\[\left\{\begin{array}{ll}
\vspace{1mm} \frac{1}{h_1}\partial_z^2\phi^0=0 & \mbox{in } \SS^+, \\
\phi^0 =0 & \mbox{on } \{z=1\}, \\
 \frac{1}{h_1}\partial_z\phi^0 =0 & \mbox{on } \{z=0\},
\end{array}
\right.\]
so that we have
\begin{equation}
 \phi^0 =0.
\end{equation}

\paragraph{At order $O(\mu)$}
\[\left\{\begin{array}{ll}
 \vspace{1mm}\frac{1}{h_1}\partial_z^2\phi^1=-\nabla_{X,z}\cdot P^1 \nabla_{X,z}\phi^0=0 & \mbox{in } \SS^+, \\
\phi^1 =0 & \mbox{on } \{z=1\}, \\
 \frac{1}{h_1}\partial_z\phi^1 =-e_{d+1}\cdot   P^1\nabla_{  X,  z}\phi^0-\nabla\cdot  (h_2\nabla\psi_2) & \mbox{on } \{z=0\},
\end{array}
\right.\]
which gives immediately
\begin{equation}
 \phi^1 =-h_1\nabla\cdot  (h_2\nabla\psi_2)(z-1).
\end{equation}

\paragraph{At order $O(\mu^2)$}
\[\left\{\begin{array}{ll}
\vspace{1mm}\frac{1}{h_1}\partial_z^2\phi^2=h_1\big((z-1)h_1\nabla \cdot\nabla \mathcal{A}_2-2\epsilon_1\nabla\zeta_1\cdot\nabla \mathcal{A}_2-\epsilon_1\Delta\zeta_1 \mathcal{A}_2 \big) & \mbox{in } \SS^+, \\
\phi^2 =0 & \mbox{on } \{z=1\}, \\
 \frac{1}{h_1}\partial_z\phi^2 =\nabla \cdot \mathcal T[h_2,\beta b]\nabla\psi_2+\epsilon_2\nabla\zeta_2\cdot(h_1\nabla \mathcal{A}_2+\epsilon_1\nabla\zeta_1 \mathcal{A}_2) & \mbox{on } \{z=0\},
\end{array}
\right.\]
with the notation $\mathcal{A}_2:=\nabla\cdot  (h_2\nabla\psi_2)$. This leads to the solution
\begin{eqnarray}
 \phi^2 =&&h_1\Big((h_1^2\nabla \cdot\nabla \mathcal{A}_2)(\frac{z^3}{6}-\frac{z^2}{2}+\frac{1}{3})-h_1(2\epsilon_1\nabla\zeta_1\cdot\nabla \mathcal{A}_2+\epsilon_1\Delta\zeta_1 \mathcal{A}_2)(\frac{z^2}{2}-\frac{1}{2})\nonumber \\&&\quad+(\nabla \cdot \mathcal T[h_2,\beta b]\nabla\psi_2+\epsilon_2\nabla\zeta_2\cdot(h_1\nabla \mathcal{A}_2+\epsilon_1\nabla\zeta_1 \mathcal{A}_2))(z-1)\Big). 
\end{eqnarray}

This formal derivation of $\bar\phi_1^{app}$ allows us to obtain the expansion of $\bar{G_1}\psi_2$, the contribution on the Dirichlet-Neumann operator from $\bar{\phi_1}$. Formally, we have
\begin{equation}
\label{G1b}
 \bar{G_1}\psi_2\approx-\mu \mathcal{A}_2+\mu^2\Big(\nabla \cdot \mathcal T[h_2,\beta b]\nabla\psi_2-\frac{1}{2}\nabla \cdot(h_1^2\nabla  \mathcal{A}_2)-\nabla\cdot(h_1\epsilon_1\nabla\zeta_1 \mathcal{A}_2) \Big).
 \end{equation}

Summing this expansion with the one of Proposition~\ref{G1c} gives immediately the expansion of the full operator $G_1(\psi_1,\psi_2)$. The following Proposition gives a rigorous statement of this fact ; its proof is postponed to Annex~\ref{PreuveG}.
\begin{proposition}
\label{G1}
Let $t_0>d/2$ and $s\geq t_0+1/2$, $\nabla \psi_1$, $\nabla\psi_2 \in H^{s+11/2}(\RR^2)$, $\zeta_1 \in H^{s+7/2}(\RR^2)$, $\zeta_2 \in H^{s+9/2}(\RR^2)$ and $b\in H^{s+11/2}(\RR^2)$, such that~\eqref{h} is satisfied. Then one has
\begin{eqnarray}
\label{G10} \big\vert G_1(\psi_1,\psi_2)+\mu (\mathcal{A}_1+\mathcal{A}_2)\big\vert_{H^s}&\leq& \mu^2 C_0,\\
\big\vert G_1(\psi_1,\psi_2)+\mu (\mathcal{A}_1+\mathcal{A}_2) -\mu^2\Big(\nabla \cdot \mathcal T_1+\nabla \cdot \mathcal T_2-\frac{1}{2}\nabla \cdot(h_1^2\nabla  \mathcal{A}_2) \ && \nonumber \\
\label{G11} -\nabla\cdot(h_1\epsilon_1\nabla\zeta_1 \mathcal{A}_2) \Big)\big\vert_{H^s}&\leq &\mu^3 C_1,
 \end{eqnarray}
with the constants \[\begin{array}{rl}
C_0&=C(\frac{1}{h},\beta \big\vert b\big\vert_{H^{s+7/2}},\epsilon_2\big\vert\zeta_2\big\vert_{H^{s+5/2}},\epsilon_1\big\vert\zeta_1\big\vert_{H^{s+3/2}},\big\vert\nabla\psi_1\big\vert_{H^{s+7/2}},\big\vert\nabla\psi_2\big\vert_{H^{s+7/2}}),\\
C_1&=C(\frac{1}{h},\beta \big\vert b\big\vert_{H^{s+11/2}},\epsilon_2\big\vert\zeta_2\big\vert_{H^{s+9/2}},\epsilon_1\big\vert\zeta_1\big\vert_{H^{s+7/2}},\big\vert\nabla\psi_1\big\vert_{H^{s+11/2}},\big\vert\nabla\psi_2\big\vert_{H^{s+11/2}}),
\end{array}\] and the notations
\[\begin{array}{cc}
 \mathcal{A}_1:=\nabla\cdot  (h_1\nabla\psi_1), &  \mathcal{A}_2:=\nabla\cdot  (h_2\nabla\psi_2), \\
\mathcal T_1:=\mathcal T[h_1,\epsilon_2\zeta_2]\nabla\psi_1, & \mathcal T_2:=\mathcal T[h_2,\beta b]\nabla\psi_2.
\end{array} \]
\end{proposition}
\begin{Rem}As in Remark~\ref{remphi}, the proof of the estimate~\eqref{G10} requires the approximate solution $\phi_1^{app,1}$, with \[\phi_1^{app,1}:=\check{\phi}_1^{app,1}+\phi^0+\mu\phi^1, \] and the second estimate~\eqref{G11} uses \[\phi_1^{app,2}:=\check{\phi}_1^{app,2}+\phi^0+\mu\phi^1+\mu^2\phi^2.\] 
In Appendix~\ref{PreuveG} (Steps 4 and 5), we give estimates on $\phi_1-\phi_1^{app}$, obtained thanks to the trace theorem and an elliptic estimate on the boundary value problem solved by $\phi_1-\phi_1^{app}$. This leads to the desired inequalities, since \[G_1(\psi_1,\psi_2)-\partial_n{\phi_1^{app}}_{\vert z=1}=\partial_n(\phi_1-\phi_1^{app})_{\vert z=1}.\]
\end{Rem}
The last expansion to obtain is the one of $H(\psi_1,\psi_2)$, which is given by the following.
\begin{proposition}
\label{PH}
Let $t_0>d/2$ and $s\geq t_0+1/2$, $\nabla \psi_1$, $\nabla\psi_2 \in H^{s+11/2}(\RR^2)$, $\zeta_1 \in H^{s+7/2}(\RR^2)$, $\zeta_2 \in H^{s+9/2}(\RR^2)$ and $b\in H^{s+11/2}(\RR^2)$, such that~\eqref{h} is satisfied. Then one has
\begin{eqnarray}
\label{H0} \big\vert H(\psi_1,\psi_2)-\nabla\psi_1\big\vert_{H^s}&\leq& \mu C_0,\\
 \big\vert H(\psi_1,\psi_2)-\nabla\psi_1-\mu\nabla\Big(h_1(\mathcal{A}_1+\mathcal{A}_2)-\frac{1}{2}h_1^2\Delta\psi_1 \quad &&\nonumber \\
 \label{H} -h_1\epsilon_1\nabla\zeta_1\cdot\nabla\psi_1\Big)\big\vert_{H^s}&\leq& \mu^2 C_1,
 \end{eqnarray}
with \[\begin{array}{rl}
C_0&=C(\frac{1}{h},\beta \big\vert b\big\vert_{H^{s+7/2}},\epsilon_2\big\vert\zeta_2\big\vert_{H^{s+5/2}},\epsilon_1\big\vert\zeta_1\big\vert_{H^{s+3/2}},\big\vert\nabla\psi_1\big\vert_{H^{s+7/2}},\big\vert\nabla\psi_2\big\vert_{H^{s+7/2}}),\\
C_1&=C(\frac{1}{h},\beta \big\vert b\big\vert_{H^{s+11/2}},\epsilon_2\big\vert\zeta_2\big\vert_{H^{s+9/2}},\epsilon_1\big\vert\zeta_1\big\vert_{H^{s+7/2}},\big\vert\nabla\psi_1\big\vert_{H^{s+11/2}},\big\vert\nabla\psi_2\big\vert_{H^{s+11/2}}),
\end{array}\] and using the notations of Proposition~\ref{G1}.
\end{proposition}
\begin{proof}
 The proof uses the estimates~\eqref{estHs} and~\eqref{estHs1} on $u:=\phi_1-\phi_1^{app,1}$. Indeed, we have to give an estimate for $\big\vert \nabla u_{\vert  z=0}\big\vert _{H^s}$, and a trace theorem (see M\'etivier~\cite{Metivier1} pp.23-27) gives for all $s\geq 0$,
\[\big\vert \nabla u_{\vert  z=0}\big\vert _{H^s}\leq  \mbox{Cst}(\big\|\Lambda^{s+1/2}\nabla u\big\|_{L^2}+\big\|\Lambda^{s-1/2}\partial_z\nabla u\big\|)_{L^2}\leq \frac{\mbox{Cst}}{\sqrt{\mu}} \big\|\Lambda^{s+1/2}\nabla^\mu_{X,z} u\big\|_{L^2}. \]
Then, the estimate~\eqref{estHs} allows to conclude:
\[\begin{array}{r}\big\vert \nabla u_{\vert  z=0}\big\vert_{H^s}\leq \frac{C_{s,t_0}}{\sqrt\mu}(\frac{1}{h},\epsilon_1\big\vert \zeta_1\big\vert _{H^{\mathrm{max}\{t_0+2,s+3/2\}}},\epsilon_2\big\vert \zeta_2\big\vert _{H^{\mathrm{max}\{t_0+2,s+3/2\}}} ))(\mu^2\big\| \hh\big\|_{H^{s+1/2}}\\+ \frac{1+\sqrt\mu}{\sqrt\mu}\big\vert V\big\vert _{H^{s+1}}).\end{array}\]
The first estimate~\eqref{H0} follows from this relation, together with the estimates~\eqref{esth} and~\eqref{estV}.

As for the Proposition~\ref{G1}, the second estimate~\eqref{H} requires the use of the higher order approximate solution  $\t u:=\phi_1-\phi_1^{app,2}$, and the result is obtained in the same way. 
\end{proof}
\begin{Rem}Using the same approximate solution as for the expansion of $G_1(\psi_1,\psi_2)$, we obtain an estimate one order less precise in $\mu$ than in~\eqref{G10} and~\eqref{G11}. This loss of precision is not seen at the formal level and comes from the $\frac{1}{\sqrt\mu}$ term, due to the horizontal scaling, which is necessary in order to have a uniformly elliptic operator.
\end{Rem}

\subsection{Asymptotic models}
\label{asymmod}
The expansions of the operators we obtained allow us to derive asymptotic models from~\eqref{sys}. The frame of this study is limited to shallow water/shallow water regimes, that is to say long waves and layers of similar depth ($\mu\ll 1$, and $\delta\sim 1$). However, the method could be extended to many different regimes, as it has been done in~\cite{BLS} with the rigid-lid assumption. As we see in \S~\ref{Section3}, we recover most of the models which have been introduced in the literature, as well as interesting new ones (the Boussinesq/Boussinesq model with coefficients~\eqref{coefc}), and the higher order system~\eqref{HO}). Furthermore, we show rigorously that~\eqref{sys} is consistent with all of these models, in the following sense (see~\cite{BCL}).
\begin{definition}
 The internal-wave system~\eqref{sys} is consistent with a system $S$ of $2d+2$ equations, if any sufficiently smooth solution of~\eqref{sys} such that~\eqref{h} is satisfied solves $S$ up to a small residual called the precision of the asymptotic model. Throughout this paper, the precision is given in the sense of $L^\infty H^s$ norms, which means that the $H^s$ norm of the residual is uniformly bounded, with respect to $t$ where the solution is defined. 
\end{definition}
\begin{Rem}
The consistency does not require the well-posedness of~\eqref{sys}, and only concerns the properties of smooth solutions of the system. However, if we assume the existence of such functions, we can prove that they are approximated by the solutions of consistent systems, as we see in \S~\ref{Solapp}.
\end{Rem}

\subsubsection{The shallow water/shallow water regime: $\mu\ll 1$} \label{secSWSW}
We assume here that both layers are in the shallow-water regime ($\mu\ll 1$), whereas strong nonlinearity are allowed ($\epsilon_1,\epsilon_2=O(1)$). We use the first order expansions~\eqref{G20},~\eqref{G10} and~\eqref{H0}, and we plug them into~\eqref{sys}. We obtain, discarding the $O(\mu)$ terms, the following system:
\begin{equation} \label{SWSW}\left\{ \begin{array}{l}
\alpha \partial_t \zeta_1 + \nabla\cdot(h_1\nabla\psi_1)+\nabla\cdot(h_2\nabla\psi_2) =  0,  \\ 
\partial_t \zeta_2 +\nabla\cdot(h_2\nabla\psi_2)  =  0,  \\
\partial_t \nabla\psi_1 + \alpha \nabla \zeta_1 +\dfrac{\epsilon_2}{2}\nabla \left(|\nabla\psi_1|^2\right) =  0, \\ 
\partial_t \nabla\psi_2 + (1-\gamma)\nabla \zeta_2 + \gamma \alpha\nabla \zeta_1 +\dfrac{\epsilon_2}{2}\nabla\left( |\nabla\psi_2|^2 \right) =  0, \end{array} \right. \end{equation}
where $h_1=1+\epsilon_1\zeta_1-\epsilon_2\zeta_2$ and $h_2=\frac{1}{\delta}-\beta b+\epsilon_2\zeta_2$.
\begin{Rem}
 This system has already been introduced in the flat bottom case in~\cite{CGK}, and equivalently, though under a different form, in~\cite{CCa2}. We say more about this in \S~\ref{sectionlayermean}.
\end{Rem}
\begin{proposition} \label{consSW}
The full system~\eqref{sys} is consistent with~\eqref{SWSW}, at the precision $\mu C_0$, with \[\begin{array}{r}C_0=C(\frac{1}{h},\beta \big\vert b\big\vert_{W^{1,\infty}H^{s+7/2}},\epsilon_2\big\vert\zeta_2\big\vert_{W^{1,\infty}H^{s+5/2}},\epsilon_1\big\vert\zeta_1\big\vert_{W^{1,\infty}H^{s+3/2}},\big\vert\nabla\psi_1\big\vert_{W^{1,\infty}H^{s+7/2}},\\\big\vert\nabla\psi_2\big\vert_{W^{1,\infty}H^{s+7/2}}).\end{array}\]
\end{proposition}
\begin{proof}
 Let $t_0>d/2$ and $s\geq t_0+1/2$. Let $U:=(\zeta_1,\zeta_2,\nabla\psi_1,\nabla\psi_2)$ be a solution of~\eqref{sys}, such that~\eqref{h} is satisfied, and $U\in \mathbf{H}^s$. It is straightforward to check that we have
\begin{equation} \left\{ \begin{array}{lr}
\multicolumn{2}{l}{\alpha \partial_t \zeta_1 + \nabla\cdot(h_1\nabla\psi_1)+\nabla\cdot(h_2\nabla\psi_2) =  \nabla\cdot(h_1\nabla\psi_1)+\nabla\cdot(h_2\nabla\psi_2)+\frac{1}{\mu}G_1(\psi_1,\psi_2),}  \\ 
\multicolumn{2}{l}{\partial_t \zeta_2 +\nabla\cdot(h_2\nabla\psi_2)  =  \nabla\cdot(h_2\nabla\psi_2)+\frac{1}{\mu}G_2\psi_2,}  \\
\multicolumn{2}{l}{\partial_t \nabla\psi_1 + \alpha \nabla \zeta_1 +\dfrac{\epsilon_2}{2}\nabla \left(|\nabla\psi_1|^2\right) =  \mu\epsilon_2\nabla\NN_1,} \\ 
\multicolumn{2}{l}{\partial_t \nabla\psi_2 + (1-\gamma)\nabla \zeta_2 + \gamma \alpha\nabla \zeta_1 +\dfrac{\epsilon_2}{2}\nabla\left( |\nabla\psi_2|^2 \right) =  \gamma\partial_t( H(\psi_1,\psi_2)-\nabla\psi_1)}\\&+\frac{\epsilon_2}{2} \gamma \nabla(|H(\psi_1,\psi_2)|^2-|\nabla\psi_1|^2)+\mu\epsilon_2 \nabla\NN_2+\mu\gamma\epsilon_2 \nabla\NN_1. \end{array} \right. \end{equation}
Except for $\partial_t( H(\psi_1,\psi_2)-\nabla\psi_1)$, the right-hand side is immediately bounded by $\mu C_0$, thanks to the estimates~\eqref{G20},~\eqref{G10} and~\eqref{H0}. The estimate on the derivative is obtained as in the following.

We use the study of Appendix~\ref{PreuveG}: we derive~\eqref{equ} with respect to $t$ on both sides and get
\begin{equation}\label{eqdtu}\left\{\begin{array}{ll}
 \nabla^\mu_{X,z}\cdot P^\mu \nabla^\mu_{X,z}(\partial_t u) =\mu^2\ \nabla_{X,z}^\mu\cdot \partial_t\hh - \nabla^\mu_{X,z}\cdot \partial_t(P^\mu) \nabla^\mu_{X,z} u& \mbox{in } \SS^+, \\
\partial_t u =0 & \mbox{on } \{z=1\}, \\
 \partial_n (\partial_t u) =\nabla\cdot \partial_t V +\mu^2 e_{d+1}\cdot \partial_t\hh-e_{d+1}\cdot \partial_t(P^\mu)\nabla^\mu_{X,z} u & \mbox{on } \{z=0\},
\end{array}
\right.\end{equation}
We now need estimates on the right-hand side of the system. Directly from the definition of $\hh$, we have 
\begin{equation}\big\|\partial_t \hh\big\|_{H^{s+3/2,1}}\leq C_0.\end{equation}
Thanks to the Step 4 of \S~\ref{Hsestimates}, we have 
\[\big\|\partial_t(P^\mu) \nabla^\mu_{X,z} u \big\|_{H^{s+3/2,1}}\leq C_0.\]
Finally, we can obtain the estimate on $\partial_t V$, using the same method as here on the lower layer:
\[\big\vert \partial_t V\big\vert_{H^s}\leq \mu^2 C(\frac{1}{h},\beta \big\vert b\big\vert_{W^{1,\infty}H^{s+5/2}},\big|(\epsilon_1\zeta_1,\epsilon_2\zeta_2,\nabla\psi_1,\nabla\psi_2)\big|_{\mathbf{H}^{s-1}}).\]
Then we use the study of Appendix~\ref{PreuveG}, and obtain the estimates of Steps 4 and 5 for $\partial_t u$, and use them as in Proposition~\ref{PH} in order to obtain the desired inequality:
\[\big\vert\partial_t( H(\psi_1,\psi_2)-\nabla\psi_1) \big\vert_{H^s}= \big\vert\nabla \partial_t u \big\vert_{H^s} \leq \mu C_0.\]
\end{proof}

\paragraph{Conservation laws} The first two equations of~\eqref{SWSW} reveal the conservation of mass, since a straightforward linear combination gives 
\begin{equation} \left\{ \begin{array}{l}
\partial_t h_1 + \epsilon_2\nabla\cdot(h_1\nabla\psi_1) =  0,  \\ 
\partial_t h_2 +\epsilon_2 \nabla\cdot(h_2\nabla\psi_2)  =  0. \end{array} \right. \end{equation}
We can play with the system to obtain other conservation laws. The conservations of total momentum and energy are given by
\begin{eqnarray*} \partial_t(\gamma h_1 u_1+h_2 u_2)+\nabla p+(\gamma h_1+h_2)\beta\nabla b+\nabla\cdot(\gamma h_1 u_1\otimes u_1+h_2 u_2\otimes u_2 )&=&0,\\
\partial_t\Big(\frac{1}{2}\big(\gamma h_1|u_1|^2+ h_2|u_2|^2\big)+p\Big)+\frac{1}{2}\nabla\cdot(\gamma h_1|u_1|^2 u_1+h_2|u_2|^2 u_2) \quad && \\
+\nabla\cdot(\gamma h_1^2  u_1 + h_2^2 u_2 + \gamma h_1 h_2 (u_1+u_2))  +(\gamma h_1 u_1+h_2 u_2)\beta\nabla b&=&0,
\end{eqnarray*}
with the notations $h_1=1+\epsilon_1\zeta_1-\epsilon_2\zeta_2$, $h_2=\frac{1}{\delta}-\beta b+\epsilon_2\zeta_2$, $u_i=\epsilon_2\nabla\psi_i$ $(i=1,2)$, and the ``pressure'' $p:=\frac{1}{2}\gamma h_1^2+\frac{1}{2}h_2^2+\gamma h_1h_2$.

\paragraph{Dispersion relations} When we calculate the linearized dispersion relations as in \S~\ref{linearized}, we obtain that $\omega_\pm^2(k)$ satisfy:
\[\omega_\pm^2(k) =\frac{1+\delta\pm\sqrt{(1-\delta)^2+4\gamma\delta}}{2\delta}|k|^2\]
This dispersion relation is not the same as the one of the full system (it corresponds to the first order of the expansion in $\mu$ of the solutions of~\eqref{lineq}), but we still have the condition $\gamma<1$, for the system to be linearly well-posed. The figure~\ref{dispsw} presents shallow water/shallow water model dispersion, compared with the dispersion of the full system, with the parameters $\mu=0.1$, $\delta=1/3$, $\gamma=2/3$.

\begin{figure}[htb]
\centering
\psfrag{k}{\begin{footnotesize}$k$                          \end{footnotesize}}
\psfrag{w}{\begin{footnotesize}$\omega$                          \end{footnotesize}}
 \includegraphics[width=0.7\textwidth]{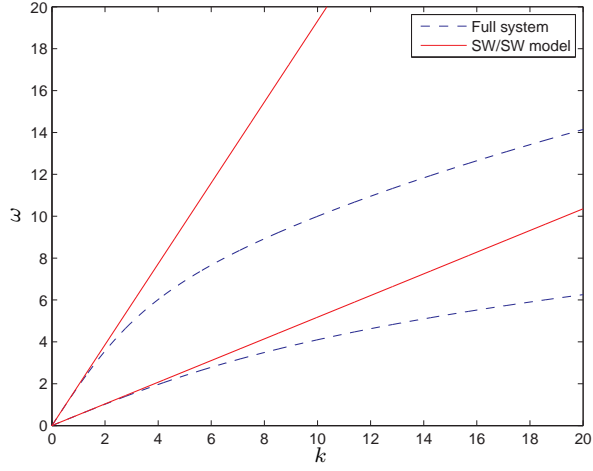}
 \caption{The shallow water/shallow water model dispersion}
\label{dispsw}
\end{figure}

\subsubsection{The Boussinesq/Boussinesq regime: $\mu \sim \epsilon_2 \sim \epsilon_1 \ll 1$}
\label{sectbouss}
In this regime, the shallowness and the nonlinearity are supposed to be small and of the same size. This time, we use the second order of the expansions, and obtain
\begin{equation}
\label{BoussBouss}
\left\{ \begin{array}{l}
\alpha\partial_t \zeta_1 + \nabla\cdot(h_1\nabla\psi_1)+\nabla\cdot(h_2\nabla\psi_2) = \mu \big(\frac{-1}{3}\Delta\nabla\cdot \nabla\psi_1 + \nabla\cdot\mathcal{T}[h_2,\beta b]\nabla\psi_2 \\ \multicolumn{1}{r}{ -\frac{1}{2\delta}\Delta\nabla\cdot \nabla\psi_2 \big) ,}  \\ 
\partial_t \zeta_2 +\nabla\cdot(h_2\nabla\psi_2) = \mu \left(\nabla\cdot\mathcal{T}[h_2,\beta b]\nabla\psi_2 \right) , \\
\partial_t \nabla\psi_1 +\alpha \nabla \zeta_1 +\dfrac{\epsilon_2}{2}\nabla \left(|\nabla\psi_1|^2\right)  = 0, \\ 
\partial_t \nabla\psi_2 + (1-\gamma)\nabla \zeta_2 + \alpha\gamma \nabla \zeta_1 +\dfrac{\epsilon_2}{2}\nabla\left( |\nabla\psi_2|^2 \right) = \mu\gamma\partial_t\big(\frac{1}{\delta}\nabla \Delta\psi_2+\frac{1}{2}\nabla\Delta\psi_1 \big),
\end{array} \right.  \end{equation}
with $\mathcal{T}[h,b]V$ defined as in Proposition~\ref{defT}.
\begin{Rem}
If the bottom is flat, then $\mathcal{T}[h_2,\beta b]\nabla\psi_2$ is simply $\frac{-1}{3\delta^3}\nabla\Delta\psi_2$.
\end{Rem}

\paragraph{Model with improved frequency dispersion}
This model is linearly ill-posed. Fortunately, following~\cite{BCL,BCS1}, we can easily derive asymptotically equivalent models, with coefficients which can be chosen so that the system is well-posed. For simplicity, we assume now to be in the case of flat bottom (see~\cite{Florent} for the varying bottom case).

We rewrite the system~\eqref{BoussBouss} with new variables: $u_i:=\nabla\phi_i(z_i)$ $(i=1,2)$. From the calculations of \S~\ref{opexpansion}, we obtain
\begin{eqnarray*}
 \phi_1^{app,1}(z)&=&\psi_1-\mu(\frac{(z-1)^2}{2}+(z-1))\Delta\psi_1-\mu\frac{1}{\delta}(z-1)\Delta\psi_2,\\
 \phi_2^{app,1}(z)&=&\psi_2-\mu\frac{1}{\delta^2}(\frac{z^2}{2}+z)\Delta\psi_2.
\end{eqnarray*}
We then define $u_1$ and $u_2$ as in the following:
\begin{eqnarray*}
 u_1&:=&\nabla\phi_1^{app,1}(z_1)=\nabla\psi_1-\mu b_1\Delta\nabla\psi_1-\mu \frac{1}{\delta}a_1\Delta\nabla\psi_2,\\
 u_2&:=&\nabla\phi_2^{app,1}(z_2)=\nabla\psi_2-\mu\frac{1}{\delta^2}a_2\Delta\nabla\psi_2,
\end{eqnarray*}
with $z_1\in (0,1)$ for the upper fluid, and $z_2\in(-1,0)$ for the lower fluid, and the coefficients 
\[a_1:=z_1-1\in[-1,0] \quad ; \quad a_2:=\frac{z_2^2}{2}+z_2\in[-1/2,0]   \quad ; \quad b_1:=\frac{a_1^2}{2}+a_1\in[-1/2,0]. \]

We plug this into~\eqref{BoussBouss} and obtain
\begin{equation}
\vspace{1mm}\label{coefc} \left\{ \begin{array}{l}
\alpha\partial_t \zeta_1 + \nabla\cdot(h_1u_1)+ \nabla\cdot(h_2u_2)+\mu\big( \frac{1+3b_1}{3}\nabla\cdot \Delta u_1\\
\multicolumn{1}{r}{+(\frac{1+2a_1}{2\delta}+\frac{1+3a_2}{3\delta^3})\nabla\cdot\Delta u_2 \big) = 0,}  \\ 
\partial_t \zeta_2 +\nabla\cdot(h_2u_2) +\mu\frac{1+3a_2}{3\delta^3}\nabla\cdot\Delta u_2= 0,\\
(1+\mu b_1 \Delta)\partial_t u_1 +\mu \frac{a_1}{\delta} \Delta\partial_t u_2 +\alpha \nabla \zeta_1 +\dfrac{\epsilon_2}{2}\nabla \left(|u_1|^2\right)  = 0, \\ 
(1+\mu (\frac{a_2}{\delta^2}-\frac{\gamma}{\delta}) \Delta)\partial_t u_2 -\mu \frac{\gamma}{2} \Delta\partial_t u_1 + (1-\gamma)\nabla \zeta_2 + \alpha\gamma \nabla \zeta_1 \\\multicolumn{1}{r}{+\dfrac{\epsilon_2}{2}\nabla\left( |u_2|^2 \right) = 0.}
\end{array} \right.  \end{equation}

\begin{Rem}
 If we choose $a_1=-\frac{1}{2}$, $a_2=-\frac{1}{3}$ and $b_1=-\frac{1}{3}$, we obtain the classical ``layer-mean'' model~\eqref{BoussBoussLM}, introduced by Choi and Camassa in~\cite{CCa2}. As we see below, this system is linearly ill-posed. One of the interests of~\eqref{coefc} is to offer a large class of equivalent models, with parameters which can be chosen so that the system is linearly well-posed.
\end{Rem}

\begin{proposition} \label{consbousscoef}
The full system~\eqref{sys} is consistent with~\eqref{coefc}, at the precision $\mu^2 C_1$, with \[\begin{array}{r}C_1=C(\frac{1}{h},\beta \big\vert b\big\vert_{W^{1,\infty}H^{s+11/2}},\epsilon_2\big\vert\zeta_2\big\vert_{W^{1,\infty}H^{s+9/2}},\epsilon_1\big\vert\zeta_1\big\vert_{W^{1,\infty}H^{s+7/2}},\big\vert\nabla\psi_1\big\vert_{W^{1,\infty}H^{s+11/2}},\\\big\vert\nabla\psi_2\big\vert_{W^{1,\infty}H^{s+11/2}}).\end{array}\]
\end{proposition}
 \begin{proof}
 Let $t_0>d/2$ and $s\geq t_0+1/2$. Let $U:=(\zeta_1,\zeta_2,\nabla\psi_1,\nabla\psi_2)$ be a solution of~\eqref{sys}, such that~\eqref{h} is satisfied, and $U\in \mathbf{H}^{s+2}$. 

We first give the proof for $a_1=b_1=a_2=0$, corresponding to the original system~\eqref{BoussBouss}. We just have to plug $U$ in~\eqref{BoussBouss}, as in the proof of Proposition~\ref{consSW}. Since $\epsilon_2\sim\mu$, we have $\big\vert \mu\epsilon_2\nabla\NN_1\big\vert_{H^s}+\big\vert \mu\epsilon_2\nabla\NN_2\big\vert_{H^s}\leq \mu^2 C_1$. The other residuals are bounded by $\mu^2 C_1$ thanks to the estimates~\eqref{G2},~\eqref{G11} and~\eqref{H} with $\epsilon_2 \ll 1$, and the equivalent estimates on the derivatives which are obtained as in the proof of Proposition~\ref{consSW}.

The general case is obtained when we substitute $\nabla\psi_1-\mu b_1\Delta\nabla\psi_1-\mu \frac{1}{\delta}a_1\Delta\nabla\psi_2$ for $u_1$, and  $\nabla\psi_2-\mu\frac{1}{\delta^2}a_2\Delta\nabla\psi_2$ for $u_2$ in~\eqref{coefc}. We obtain~\eqref{BoussBouss} up to additional terms that are clearly bounded by $\mu^2 C_1$.
\end{proof}

\paragraph{Dispersion relations} As we have said previously, the coefficients can be chosen so that the system~\eqref{coefc} is linearly well-posed. Indeed, it is straightforward to check from the linearized system that $\omega_\pm^2(k)$, corresponding to plane-wave solutions $e^{ik\cdot X-i\omega_\pm t}$, must be the solutions of the equation
\begin{equation} 
\label{quadeq}
\omega^4-A(\mu |k|^2)|k|^2\omega^2+B(\mu |k|^2)|k|^4=0,
\end{equation}
with \[\begin{array}{l}
A(Y):=\frac{(1-\beta_1 Y)(1+\frac{\gamma\delta(a_1+1)-a_2}{\delta^2} Y)+\gamma(\frac{1}{\delta}-(\alpha_1+\alpha_2) Y)(1-(b_1+\frac{1}{2}) Y)+(1-\gamma)(\frac{1}{\delta}-\alpha_2 Y)(1-b_1 Y)}{(1-b_1 Y)(1-\frac{a_2-\gamma\delta}{\delta^2} Y)+\frac{\gamma}{2\delta} a_1 Y^2},\\
B(Y):=(1-\gamma)\frac{(\frac{1}{\delta}-\alpha_2 Y)(1-\beta_1 Y)}{(1-b_1Y)(1-\frac{a_2-\gamma\delta}{\delta^2}Y)+\frac{\gamma}{2\delta} a_1 Y^2},
\end{array}\]
and the notations
\[\alpha_1:=\frac{1+2 a_1}{2\delta} \qquad;\qquad \alpha_2:=\frac{1+3a_2}{3\delta^3} \qquad ; \qquad \beta_1:=\frac{1+3b_1}{3}.\] 

In order to have two positive solutions of~\eqref{quadeq}, the coefficients have to satisfy $a_2\leq -1/3$, and $b_1\leq -1/2$. We see that the original system~\eqref{BoussBouss}, as well as the classical layer-mean model~\eqref{BoussBoussLM} are ill-posed. However, there exists sets of parameters $a_1$, $a_2$, $b_1$ such that the generalized system is well-posed. Moreover, we can choose the coefficients such that the dispersions meet with the ones of the full system, at the order 3 in $\mu|k|^2$. We present in figure~\ref{dispcoef} the difference between the dispersion of the full system and the one of the Boussinesq/Boussinesq model for three sets of parameters: $a_1=b_1=a_2=0$ corresponding to the original system~\eqref{BoussBouss}, $a_1=-\frac{1}{2}$, $a_2=-\frac{1}{3}$ and $b_1=-\frac{1}{3}$ corresponding to the layer-mean system~\eqref{BoussBoussLM}, and finally $a_1\approx0.4714$, $a_2\approx-0.3942$ and $b_1=-1$ corresponding to optimized parameters in~\eqref{coefc}. Moreover, we chose $\mu=0.1$, $\delta=1/3$, and $\gamma=2/3$. Note that except for the last set of parameters, the system is linearly ill-posed, so that the computation breaks for high wave numbers.
\begin{figure}[htb]
\centering
\psfrag{k}{\begin{footnotesize}$k$                          \end{footnotesize}}
\psfrag{wp}{\begin{footnotesize}$\omega_+$                          \end{footnotesize}}
\psfrag{wm}{\begin{footnotesize}$\omega_-$                          \end{footnotesize}}
 \includegraphics[width=0.49\textwidth]{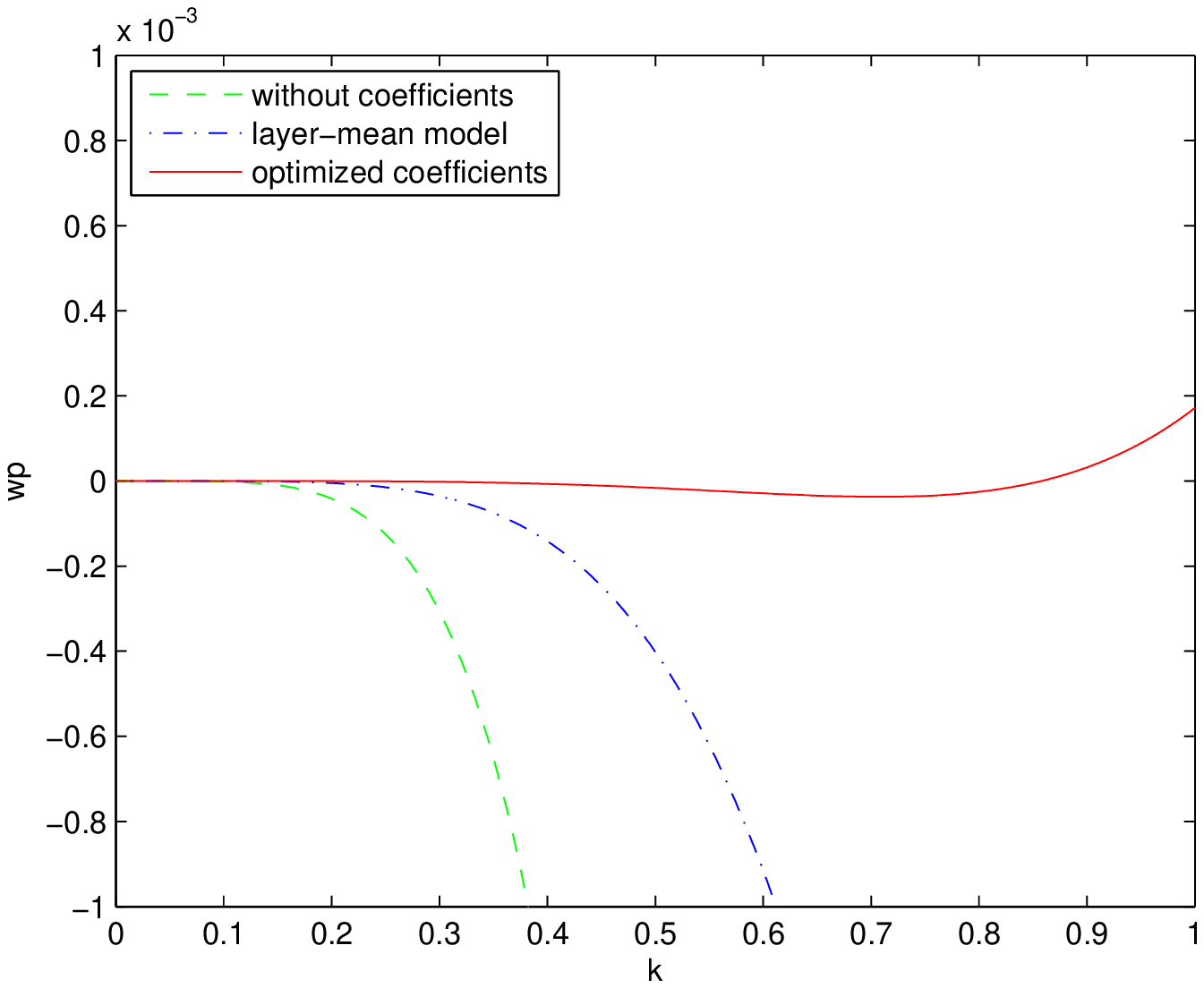} \includegraphics[width=0.49\textwidth]{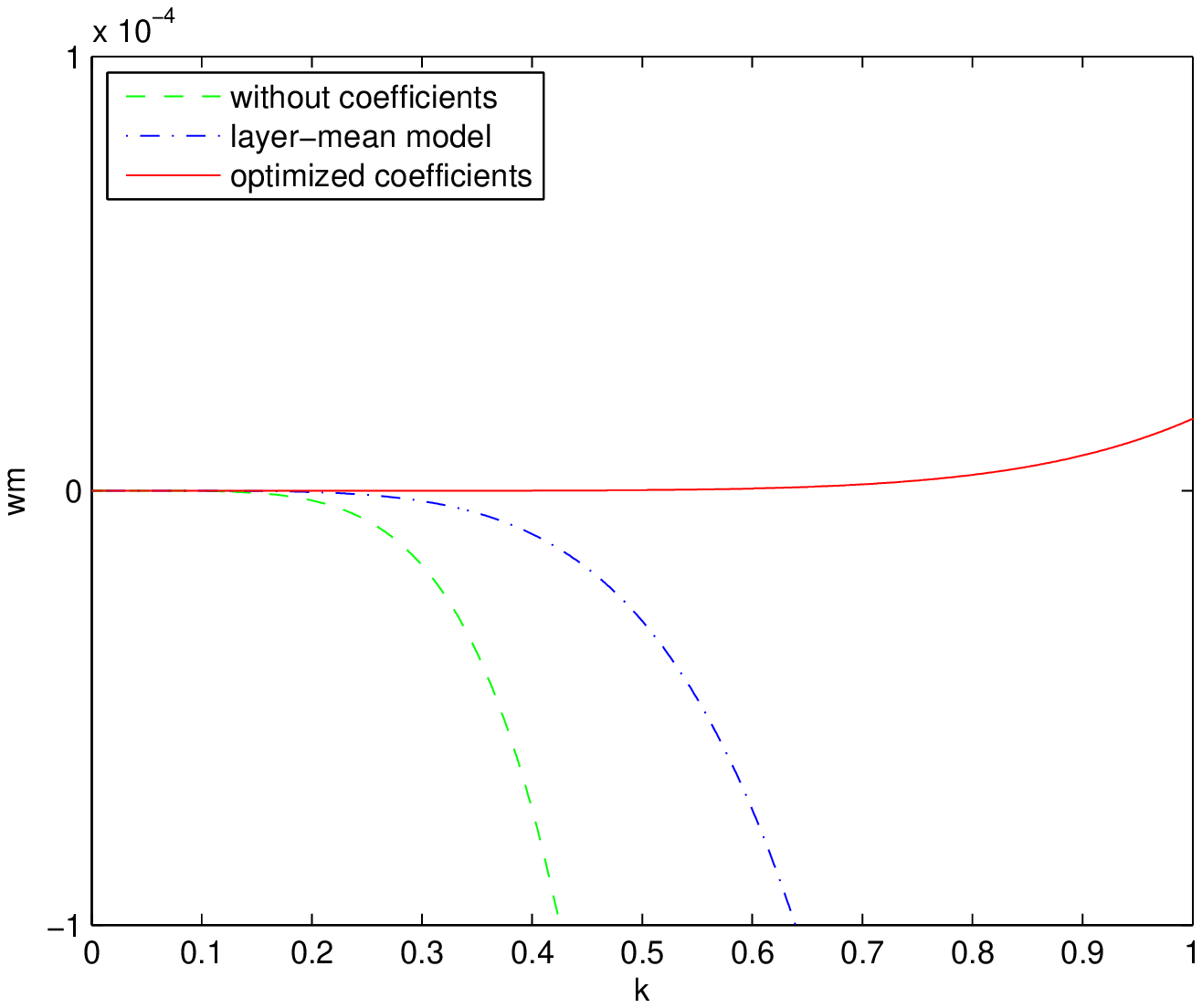}
 \caption{The Boussinesq/Boussinesq models dispersion error}
\label{dispcoef}
\end{figure}

\subsubsection{The higher order system}
We are now back in the strong linearity regime, allowing large amplitude ($\epsilon_1$, $\epsilon_2 =O(1)$). But now we use the higher order expansions~\eqref{G2},~\eqref{G11} and~\eqref{H}, and thus obtain the strongly nonlinear model
\begin{equation}\label{HO} \left\{ \begin{array}{lr}
\multicolumn{2}{l}{\alpha\partial_t \zeta_1 + \mathcal{A}_1+\mathcal{A}_2 = \mu \Big(\nabla\cdot\mathcal{T}_1 + \nabla\cdot\mathcal{T}_2 -\frac{1}{2}\nabla\cdot (h_1^2\nabla \mathcal{A}_2)-\nabla\cdot\big(h_1 \epsilon_1\nabla\zeta_1  \mathcal{A}_2)  \Big),} \\ 
\multicolumn{2}{l}{\partial_t \zeta_2 +\mathcal{A}_2 = \mu \nabla\cdot\mathcal{T}_2 ,} \\
\multicolumn{2}{l}{\partial_t \nabla\psi_1 + \alpha \nabla \zeta_1 +\dfrac{\epsilon_2}{2}\nabla \left(|\nabla\psi_1|^2\right)  =  \mu \epsilon_2\nabla\NN_1 ,}\\ 
\multicolumn{2}{l}{\partial_t \nabla\psi_2 + (1-\gamma)\nabla \zeta_2 + \gamma \alpha\nabla \zeta_1 +\dfrac{\epsilon_2}{2}\nabla\left( |\nabla\psi_2|^2 \right) = \mu\Big(\gamma\partial_t\nabla \mathcal{H}+\gamma\epsilon_2\nabla(\nabla\psi_1\cdot \nabla\mathcal{H})}\\ & +\epsilon_2\nabla \NN_2+\gamma \epsilon_2 \nabla \NN_1 \Big),
\end{array} \right. \end{equation}
where we have used the following notations:
\[\begin{array}{ll}
 \mathcal{A}_1:=\nabla\cdot  (h_1\nabla\psi_1), & \mathcal{A}_2:=\nabla\cdot  (h_2\nabla\psi_2),\\
\mathcal T_1:=\mathcal T[h_1,\epsilon_2\zeta_2]\nabla\psi_1, & \mathcal T_2:=\mathcal T[h_2,\beta b]\nabla\psi_2,\\
\multicolumn{2}{l}{\mathcal{H}:=h_1(\nabla\cdot(h_1\nabla\psi_1)+\nabla\cdot(h_2\nabla\psi_2)-\frac{1}{2}h_1\Delta\psi_1-\epsilon_1\nabla\zeta_1\cdot \nabla\psi_1),}\\
\multicolumn{2}{l}{\NN_1:= \frac{(\epsilon_1\nabla\zeta_1\cdot\nabla\psi_1-\nabla\cdot(h_1\nabla\psi_1)-\nabla\cdot(h_2\nabla\psi_2))^2}{2},}\\
\multicolumn{2}{l}{\NN_2:= \frac{(\epsilon_2\nabla\zeta_2\cdot\nabla\psi_2-\nabla\cdot(h_2\nabla\psi_2))^2-\gamma(\epsilon_2\nabla\zeta_2\cdot\nabla\psi_1-\nabla\cdot(h_2\nabla\psi_2))^2}{2}.}
\end{array}\]

\begin{proposition} \label{consHO}
The full system~\eqref{sys} is consistent with~\eqref{HO}, at the precision $\mu^2 C_1$, with \[\begin{array}{r}C_1=C(\frac{1}{h},\beta \big\vert b\big\vert_{W^{1,\infty}H^{s+11/2}},\epsilon_2\big\vert\zeta_2\big\vert_{W^{1,\infty}H^{s+9/2}},\epsilon_1\big\vert\zeta_1\big\vert_{W^{1,\infty}H^{s+7/2}},\big\vert\nabla\psi_1\big\vert_{W^{1,\infty}H^{s+11/2}},\\\big\vert\nabla\psi_2\big\vert_{W^{1,\infty}H^{s+11/2}}).\end{array}\]
\end{proposition}
 \begin{proof}
 Let $t_0>d/2$ and $s\geq t_0+1/2$. Let $U:=(\zeta_1,\zeta_2,\nabla\psi_1,\nabla\psi_2)$ be a solution of~\eqref{sys}, such that~\eqref{h} is satisfied, and $U\in \mathbf{H}^{s+2}$. 
We plug $U$ in~\eqref{HO}, and thanks to the estimates~\eqref{G2},~\eqref{G11} and~\eqref{H}, and the equivalent estimates on the derivatives are obtained as in the proof of Proposition~\ref{consSW}, we can check that the residuals are bounded by $\mu^2 C_1$.
\end{proof}

\paragraph{Dispersion relations} The linearized system is the same as the one of~\eqref{BoussBouss}. So the system is linearly ill-posed, and we should derive models with parameters, to obtain well-posed systems.

\section{Convergence results}
\label{Solapp}
We show here how to use the consistency results obtained in Section~\ref{asymmod} to prove convergence results, stating that solutions of~\eqref{sys} - if they exist - remain close to the solutions of the asymptotic models that are symmetrizable hyperbolic system. 

\begin{Rem}It is not clear that each of our models can be written as a symmetrizable hyperbolic system. That is why we focus here on the shallow water/shallow water model~\eqref{SWSW}, in the flat-bottom case ($\beta=0$). We set $d=2$, and the case $d=1$ follows immediately. The case of the Boussinesq/Boussinesq models will be discussed in a later work. 
\end{Rem}

The analysis is based on classical results for quasilinear systems, which can be found for example in~\cite{Metivier2} and~\cite{Kato}, and that we recall here.
\begin{lemma} \label{estenergie}
 Let $s>\frac{d}{2}+1$ and $T>0$. We assume that $A_j$ are smooth functions of $u\in \RR^n$, such that the system 
 \begin{equation}\label{simpeqF}\partial_t U + \sum_{j=1}^d A_j(U)\partial_x U=F(t,x,U)\end{equation}
 is Friedrichs-symmetrizable. Moreover, we assume that $u \mapsto F(t,x,u)$ is a smooth function of $u\in \RR^n$, and that $F(t,x,u)$ is bounded in $H^s$,uniformly with respect to $t\in[0,T]$. Then for $g\in H^s(\RR^d)$, taking values in $\RR^n$, there exists $0<T'\leq T$ and a unique $U\in C^0([0,T');H^s(\RR^d))^n$ such that $U$ satisfies~\eqref{simpeqF} and $U(t=0)=g$. Moreover, $U$ belongs to $U\in C^0([0,T');H^{s})^n \cap C^1([0,T');H^{s-1})^n$, and if $U$ satisfies 
\[\big\vert U\big\vert_{W^{1,\infty}([0,T]\times \RR^d)}\leq M,\]
for $M>0$, then there are constants $C(M)$ and $K(M)$ such that
\[\big\vert U(t)\big\vert_{H^s}\leq Ce^{Kt} \big\vert g\big\vert_{H^s}+C\int_0^t e^{K(t-s)}\big\vert f(s)\big\vert_{H^s} ds,\]
with $f(t,x)=F(t,x,g)$.
\end{lemma}

First, we remark that the shallow water/shallow water model~\eqref{SWSW}, in the flat-bottom case ($\beta=0$), can be written as a quasilinear system: 
\begin{equation}\label{simpeq}\partial_t U + A_1(U)\partial_x U+A_2(U)\partial_y U=0,\end{equation}
with the notation \begin{eqnarray*}U&:=&(h_1,h_2,u_{1x},u_{1y},u_{2x},u_{2y})\\
&=&(1+\epsilon_1\zeta_1-\epsilon_2\zeta_2,\frac{1}{\delta}+\epsilon_2\zeta_2,\epsilon_2\partial_x\psi_1,\epsilon_2\partial_y\psi_1,\epsilon_2\partial_x\psi_2,\epsilon_2\partial_y\psi_2),\end{eqnarray*} 
and the matrices 
\begin{eqnarray*}A_1(U)&:=&\begin{pmatrix}
 u_{1x} & 0 & h_1 & 0 & 0 & 0  \\
 0 & u_{2x} & 0 & 0 & h_2 & 0  \\
 1 & 1 & u_{1x} & u_{1y} & 0 & 0  \\
0&0&0&0&0&0\\
\gamma & 1 & 0 & 0 & u_{2x} & u_{2y} \\
0&0&0&0&0&0 \end{pmatrix} , \\
A_2(U)&:=&\begin{pmatrix}
 u_{1y} & 0 & 0 & h_1 & 0 & 0  \\
 0 & u_{2y} & 0 & 0 & 0 & h_2  \\
 0&0&0&0&0&0\\
 1 & 1 & u_{1x} & u_{1y} & 0 & 0  \\
 0&0&0&0&0&0\\
\gamma & 1 & 0 & 0 & u_{2x} & u_{2y}
\end{pmatrix}.\end{eqnarray*}

We prove now that the Cauchy problem associated with~\eqref{simpeq} is well-posed under some assumptions on the initial data, since the quasilinear system is Friedrichs-symmetrizable.
\begin{proposition}
 Let $s>\frac{d}{2}+1$. Let $U_0\in H^s(\RR^d)^6$, such that there exists $h>0$ such that for all $X$ in $\RR^d$, $U_0(X)$ satisfies the assumptions 
\begin{equation}\label{assumptions}h_1, h_2 > h, \ |u_{1x}^2+u_{1y}^2|,|u_{2x}^2+u_{2y}^2|< h, \ \text{ and }  (h_1-u_{1x}^2-u_{1y}^2)(h_2-u_{2x}^2-u_{2y}^2)> \gamma h_1h_2.\end{equation}
Then there exists $T'>0$ and a unique $U\in C^0([0,T');H^s(\RR^d))^6$ such that $U$ satisfies~\eqref{simpeq} and $U(t=0)=U_0$.
\end{proposition}
\begin{proof}
We introduce the following matrix $S$, namely
\[S(U):=\left(\begin{array}{cccccc}
 \gamma & \gamma & \gamma u_{1x} & \gamma u_{1y} & 0 & 0 \\
 \gamma &  1 & 0 & 0 & u_{2x} & u_{2y} \\
\gamma u_{1x} & 0 & \gamma h_1 & 0 & 0 & 0 \\
\gamma u_{1y} & 0 & 0 & \gamma h_1 & 0 & 0 \\
0 & u_{2x} & 0 & 0 & h_2 & 0 \\
0 & u_{2y} & 0 & 0 & 0 & h_2
\end{array}\right).\]

It is straightforward to check that $S(U)$ and $S(U)A(U,\xi)$ are self-adjoint, with $A(U,\xi):=\xi_1A_1(U)+\xi_2A_2(U)$. Then, using the Gauss reduction algorithm, one can check that $S(U)$ is definite positive if $U$ satisfies~\eqref{assumptions}. These requirements are satisfied at time $t=0$ by $U_0$, and we define $T$ as the maximum time such that they remain satisfied for all $t<T$. We know that $T>0$ thanks to a continuity argument. Then since we have proved that $S$ is a symmetrizer of~\eqref{simpeq}, Lemma~\ref{estenergie} gives $0<T'\leq T$ such that $U$ is uniquely defined on $[0,T')$.\qquad
\end{proof}

The last step consists in proving that the solutions of~\eqref{simpeq} approximate the solutions of the full system~\eqref{sys}, assuming that the latter exist. This is obtained thanks to the energy estimate of Lemma~\ref{estenergie}.
\begin{proposition} We fix $\gamma\in(0,1)$ and $\delta\in (0,+\infty)$. For $t_0>d/2$ and $s\geq t_0+1/2$, let $U\in C^1([0;T];H^s)^6\cap C^0([0;T];H^{s+1})^6$ be a solution of~\eqref{sys} such that~\eqref{h} is satisfied and $U$ is bounded in $\mathbf{H}^s([0,T])$, uniformly with respect to $\epsilon_1$, $\epsilon_2\in [0,1)$, and $\mu\in (0,\mu^{max}]$. We denote by $\t U:=(\t\zeta_1,\t\zeta_2,\t u_1,\t u_2)$ the solution of~\eqref{SWSW}, with the same initial values, that we assume to satisfy~\eqref{assumptions}. 
Then one has \[\big|U-\t U\big|_{{H}^s}\leq \mu C_0,\]
with $C_0=C(\frac{1}{h},\gamma,\delta,\mu^{max},\big|U\big|_{\mathbf{H}^s},T)$.
\end{proposition}

{\it Proof.} 
Thanks to the consistency result (Proposition~\ref{consSW}), we know that $U$ satisfies~\eqref{simpeqF}, with $F(t,x,U)=f(t,x)$ and
\[\big\vert f\big\vert_{H^s}\leq \mu C_0,\] with $C_0=C(\frac{1}{h},\big|U\big|_{\mathbf{H}^s})$. 
Then, the difference between the two solutions $R^\mu:=U-\t U$ satisfies~\eqref{simpeqF}, with the same $f$ and
\[F(t,x,R^\mu):=f(t,x)-A_1(R^\mu)\partial_x \t U-A_2(R^\mu)\partial_y \t U.\]
Taking a smaller $T$ if necessary, one has \[\big\vert U\big\vert_{(W^{1,\infty}([0,T]\times \RR^d))^6}+\big\vert \t U\big\vert_{(W^{1,\infty}([0,T]\times \RR^d))^6}\leq M,\]
where $M$ is independent of $\epsilon_1$, $\epsilon_2$ and $\mu$.
Thus, we can apply Lemma~\ref{estenergie}, and one has 
\[\big\vert R^\mu(t)\big\vert_{H^s}\leq C\int_0^t e^{K(t-s)}\big\vert f \big\vert_{H^s} ds\leq \mu C(\frac{1}{h},\gamma,\delta,\mu^{max},\big|U\big|_{\mathbf{H}^s},T).\qquad\endproof\]

\section{Links to other models}
\label{Section3}
\subsection{Rigid lid in the shallow water/shallow water case}\label{rigidlid}
In~\cite{BLS}, Bona, Lannes and Saut presented a model for internal waves in the shallow water regime, with the rigid lid assumption. They showed that a nonlocal operator has to appear for $d=2$ (see observations in~\cite{GLS}). This operator cannot be seen in our model~\eqref{SWSW}, so that it is a purely two dimensional, rigid lid effect. However, we show in the following how to make it appear from~\eqref{SWSW}.

Indeed, the rigid lid assumption means that $\epsilon_1=0$, when $\epsilon_2$ remains $>0$, so that $\alpha=0$. The system~\eqref{SWSW} becomes
\begin{equation} \label{SWSWrl}\left\{ \begin{array}{l}
\nabla\cdot(h_1\nabla\psi_1)+\nabla\cdot(h_2\nabla\psi_2) =  0,  \\ 
\partial_t \zeta_2 +\nabla\cdot(h_2\nabla\psi_2)  =  0,  \\
\partial_t \nabla\psi_1 +\dfrac{\epsilon_2}{2}\nabla \left(|\nabla\psi_1|^2\right) =  0, \\ 
\partial_t \nabla\psi_2 + (1-\gamma)\nabla \zeta_2 +\dfrac{\epsilon_2}{2}\nabla\left( |\nabla\psi_2|^2 \right) =  0, \end{array} \right. \end{equation}
where $h_1=1-\epsilon_2\zeta_2$ and $h_2=\frac{1}{\delta}-\beta b+\epsilon_2\zeta_2$.

For simplicity, we restrict ourself to the case of a flat bottom ($\beta=0$), but we could do the same calculations with $\beta>0$.
We first define the shear velocity
\[v:=\nabla\psi_2-\gamma\nabla\psi_1.\]
From the first line, we deduce:
\[\nabla\cdot(h_2 v)=-\nabla\cdot((h_1+\gamma h_2)\nabla\psi_1)=-\frac{\gamma+\delta}{\delta}\nabla\cdot((1+\frac{\gamma-1}{\gamma+\delta}\delta\epsilon_2\zeta_2)\nabla\psi_1).\]
Then we define the nonlocal operator $\mathfrak Q$ as follows:
\begin{definition}
Assuming that $\zeta \in L^\infty (\RR^d)$, we define the mapping
\[ \mathfrak Q[\zeta]:= \begin{array}{ccc}
 L^2(\RR^d)^d & \rightarrow & L^2(\RR^d)^d \\
 W & \mapsto & V
\end{array}
\]
where $V$ is the unique gradient vector in $L^2(\RR^d)^d$, solution of the equation
\[\nabla\cdot((1+\zeta)V)=\nabla \cdot W.\]
\end{definition}
So from the definition, we have 
\[\nabla\psi_1=\mathfrak{Q}[\frac{\gamma-1}{\gamma+\delta}\delta\epsilon_2\zeta_2](-\frac{\delta}{\gamma+\delta}h_2v).\]
We plug this expression into~\eqref{SWSWrl}, and obtain immediately
\begin{equation} \label{SWSWRL}\left\{ \begin{array}{l} 
\partial_t \zeta_2 +\frac{\delta}{\gamma+\delta}\nabla\cdot\big(h_1 \mathfrak{Q}[\frac{\gamma-1}{\gamma+\delta}\delta\epsilon_2\zeta_2](h_2v) \big)=  0,  \\
\partial_t v + (1-\gamma)\nabla \zeta_2 +\dfrac{\epsilon_2}{2}\nabla\left( |v-\frac{\gamma\delta}{\gamma+\delta}\mathfrak{Q}[\frac{\gamma-1}{\gamma+\delta}\delta\epsilon_2\zeta_2](h_2v)|^2\right.\\ \multicolumn{1}{r}{ \left. -\frac{\gamma\delta^2}{(\gamma+\delta)^2}|\mathfrak{Q}[\frac{\gamma-1}{\gamma+\delta}\delta\epsilon_2\zeta_2](h_2v)|^2\right) =  0,} \end{array} \right. \end{equation}
where $h_1=1-\epsilon_2\zeta_2$ and $h_2=\frac{1}{\delta}+\epsilon_2\zeta_2$. This is exactly the system derived in~\cite{BLS}.

Using the same method, we could derive rigid-lid models from~\eqref{coefc} and~\eqref{HO}. The rigid-lid model in the Boussinesq regime has already been exhibited in~\cite{BLS}, and a fully nonlinear model is presented in~\cite{CCa}.

\subsection{The layer-mean equations}
\label{sectionlayermean}
In the literature, the water-wave system is often given by layer-mean equations (see for example~\cite{Wu}), using as unknowns the depth-mean velocity across the layers: 
\begin{eqnarray*}
 \overline{u}_1(X)&:=\frac{1}{h_1}\int_{\epsilon_2\zeta_2}^{1+\epsilon_1\zeta_1}\nabla\phi_1(X,r_1(X,z))dz &\mbox{ with } h_1:=1+\epsilon_1\zeta_1-\epsilon_2\zeta_2, \\
 \overline{u}_2(X)&:=\frac{1}{h_2}\int_{-1/\delta+\beta b}^{\epsilon_2\zeta_2}\nabla\phi_2(X,r_2(X,z))dz &\mbox{ with } h_2:=\frac{1}{\delta}-\beta b+\epsilon_2\zeta_2. 
\end{eqnarray*}
The systems under this form (obtained for example in~\cite{CCa2} and~\cite{BGT}) are equivalent to the system we derived, since one can approximate $\overline{u}_1$ and $\overline{u}_2$ thanks to our previous unknowns $\psi_1$ and $\psi_2$ (as we see in the following Proposition), and conversely. Thus, our study gives a rigorous justification of these models, and we are able to offer consistency results.
\begin{proposition}
 \label{uvsphi}
Let $t_0>d/2$ and $s\geq t_0+1/2$, $\nabla \psi_1$, $\nabla\psi_2 \in H^{s+11/2}(\RR^2)$, $\zeta_1 \in H^{s+7/2}(\RR^2)$, $\zeta_2 \in H^{s+9/2}(\RR^2)$ and $b\in H^{s+11/2}(\RR^2)$, such that~\eqref{h} is satisfied. Then one has
\begin{eqnarray}
\label{h1u10} \big\vert \overline{u}_1-\nabla\psi_1 \big\vert_{H^{s+1}}&\leq &\mu C_0,\\
\label{h2u20} \big\vert \overline{u}_2-\nabla\psi_2 \big\vert_{H^{s+1}}&\leq& \mu C_0,\\
\label{h1u11} \big\vert \overline{u}_1-\nabla\psi_1-\mu\mathcal D_1 (\nabla\psi_1,\nabla\psi_2) \big\vert_{H^{s+1}}&\leq& \mu^2 C_1,\\
\label{h2u21} \big\vert \overline{u}_2-\nabla\psi_2-\mu\mathcal D_2 \nabla\psi_2 \big\vert_{H^{s+1}}&\leq& \mu^2 C_1,
 \end{eqnarray}
with \[\begin{array}{r}C_j=C(\frac{1}{h},\beta \big\vert b\big\vert_{H^{s+7/2+2j}},\epsilon_2\big\vert\zeta_2\big\vert_{H^{s+5/2+2j}},\epsilon_1\big\vert\zeta_1\big\vert_{H^{s+3/2+2j}},\big\vert\nabla\psi_1\big\vert_{H^{s+7/2+2j}},\\\big\vert\nabla\psi_2\big\vert_{H^{s+7/2+2j}}),\end{array}\] and where $\mathcal D_1$ and $\mathcal D_2$ are defined by
\[\begin{array}{rl}
\mathcal D_1 (\nabla\psi_1,\nabla\psi_2)&= -\frac{1}{h_1}\Big(\mathcal T_1-\frac{1}{2}(h_1^2\nabla  \mathcal{A}_2)-(h_1\epsilon_1\nabla\zeta_1 \mathcal{A}_2)\Big),\\
\mathcal D_2 \nabla\psi_2&= -\frac{1}{h_2}\mathcal T_2,
\end{array}\]
with the notations of Proposition~\ref{G1}.
\end{proposition}
\begin{proof}
Using the Green formula with $\phi_1$ the solution of~\eqref{phi1}, and a test function $\t \varphi:=(X,z)\mapsto \varphi(X)$, we have
\begin{eqnarray*} 
\int_{\Omega_1}\t \varphi\Delta^\mu_{X,z}\phi_1 dX dz&=&-\int_{\Omega_1}\nabla^\mu_{X,z} \phi_1\cdot\nabla^\mu_{X,z} \t \varphi dX dz+\int_{\Gamma_1}\varphi\partial_{n_1}\phi_1 dn_1+\int_{\Gamma_2}\varphi\partial_{n_2}\phi_1 dn_2\\
&=&-\mu\int_{\RR^d} \nabla\varphi \int_{\epsilon_2\zeta_2}^{1+\epsilon_1\zeta_1}\nabla\phi_1dz dX +\int_{\RR^d}\varphi\left( G_1(\psi_1,\psi_2)-G_2\psi_2\right)  dX .
\end{eqnarray*}
Thus, we deduce 
\begin{equation}
\label{u1G1} \nabla\cdot(h_1\overline{u}_1)=\frac{-1}{\mu}(G_1(\psi_1,\psi_2)-G_2\psi_2).
\end{equation}
Identically, we have
\begin{equation}
\label{u2G2} \nabla\cdot(h_2\overline{u}_2)=\frac{-1}{\mu}G_2\psi_2.
\end{equation}
We now prove the estimate~\eqref{h1u10}, and the others are obtained in the same way.

Using the Propositions~\ref{defT} and~\ref{G1} together with~\eqref{u1G1}, and since~\eqref{h} is satisfied, one has immediately
\[\big\vert \nabla\cdot(\overline{u}_1-\nabla\psi_1) \big\vert_{H^{s}}\leq \mu C_0,\]
so that we only have to obtain an $L^2$-estimate on $\overline{u}_1-\nabla\psi_1$. Using the definition of $\overline{u}_1$ and the mappings defined on \S~\ref{asymptexp}, we obtain
\[\overline{u}_1-\nabla\psi_1=\int_0^1 \nabla(\t\phi_1-\psi_1)+\nabla s_1\partial_{\t z} \t\phi_1 d\t z,\]
with $\phi_1:(X,\t z)\in \SS^+\mapsto \t \phi_1(X,s_1(X,\t z))$. We deduce \[\big\vert \overline{u}_1-\nabla\psi_1 \big\vert_2\leq C(\epsilon_1\big\vert \zeta_1\big\vert _{W^{1,\infty}},\epsilon_2\big\vert \zeta_2\big\vert _{W^{1,\infty}})\big\| \nabla_{X,\t z}(\t\phi_1-\psi_1)  \big\|_2 .\] The estimate follows now from Step 3 of \S~\ref{Hsestimates}, together with the estimates~\eqref{esth} and~\eqref{estV}.
\end{proof}

\subsubsection{The shallow water/shallow water regime: $\mu\ll 1$, $\epsilon=O(1)$}
We use~\eqref{h1u10} and~\eqref{h2u20} in the system~\eqref{SWSW}, and with a straightforward linear combination, we obtain
\begin{equation} \label{SWSWLM}\left\{ \begin{array}{l}
\partial_t h_1 + \epsilon_2\nabla\cdot(h_1\overline{u}_1) =  0,  \\ 
\partial_t h_2 +\epsilon_2\nabla\cdot(h_2\overline{u}_2)  =  0,  \\
\partial_t \overline{u}_1 + \nabla h_2+\beta\nabla b +\nabla h_1 +\dfrac{\epsilon_2}{2}\nabla \left(|\overline{u}_1|^2\right) =  0, \\ 
\partial_t \overline{u}_2 + \nabla h_2+\beta\nabla b + \gamma \nabla h_1 +\dfrac{\epsilon_2}{2}\nabla\left( |\overline{u}_2|^2 \right) =  0. \end{array} \right. \end{equation}
\begin{proposition}
The full system~\eqref{sys} is consistent with~\eqref{SWSWLM}, at the precision $\mu C_0$, with \[\begin{array}{r}C_0=C(\frac{1}{h},\beta \big\vert b\big\vert_{W^{1,\infty}H^{s+7/2}},\epsilon_2\big\vert\zeta_2\big\vert_{W^{1,\infty}H^{s+5/2}},\epsilon_1\big\vert\zeta_1\big\vert_{W^{1,\infty}H^{s+3/2}},\big\vert\nabla\psi_1\big\vert_{W^{1,\infty}H^{s+7/2}},\\\big\vert\nabla\psi_2\big\vert_{W^{1,\infty}H^{s+7/2}}).\end{array}\]
\end{proposition}
\begin{proof}
We know from Proposition~\ref{consSW} that~\eqref{sys} is consistent with~\eqref{SWSW}, at the precision $\mu C_0$. From~\eqref{h1u10} and~\eqref{h2u20}, we deduce that $(\zeta_1,\zeta_2,\overline{u}_1,\overline{u}_2)$ satisfies~\eqref{SWSWLM} up to a residual of the same order.
\end{proof}
\begin{Rem}
 Note that the first two equations of~\eqref{SWSWLM} are equalities (where the last two equations are first order approximations in $\mu$), as we can see from~\eqref{sys},~\eqref{u1G1} and~\eqref{u2G2}. They reveal the conservation of mass. Conservation of momentum and energy are the one obtained in \S~\ref{secSWSW}, when we substitute $\overline{u}_i$ for $\nabla\psi_i$ $(i=1,2)$. These conservation laws, and the one of higher order systems, have already been introduced in the flat-bottom case in~\cite{BGT}.
\end{Rem}

\subsubsection{The Boussinesq/Boussinesq regime: $\mu \sim \epsilon_2 \sim \epsilon_1 \ll 1$}
We now restrict ourself to the flat-bottom case, since it considerably simplifies the notations, but the following could be derived with $\beta\neq 0$ without any difficulty. The estimates~\eqref{h1u11} and~\eqref{h2u21} with $\epsilon_2 \sim \mu$ and $\beta=0$ give the following formal relations
\begin{eqnarray}
 \overline{u}_1&\approx&\nabla\psi_1+\mu(\frac{1}{3}\nabla\Delta\psi_1+\frac{1}{2\delta}\nabla\Delta\psi_2),\\
 \overline{u}_2&\approx&\nabla\psi_2+\mu\frac{1}{3\delta^2}\nabla\Delta\psi_2.
\end{eqnarray}
Plugging this into~\eqref{BoussBouss} we obtain the system
\begin{equation}
\label{BoussBoussLM}
\left\{ \begin{array}{l}
\partial_t h_1 + \epsilon_2\nabla\cdot(h_1\overline{u}_1) =  0,  \\ 
\partial_t h_2 +\epsilon_2\nabla\cdot(h_2\overline{u}_2)  =  0, \\
\partial_t \overline{u}_1 +\alpha \nabla \zeta_1 +\dfrac{\epsilon_2}{2}\nabla \left(|\overline{u}_1|^2\right)  = \mu\partial_t(\frac{1}{3}\Delta\overline{u}_1+\frac{1}{2\delta}\Delta\overline{u}_2), \\ 
\partial_t \overline{u}_2 + (1-\gamma)\nabla \zeta_2 + \alpha\gamma \nabla \zeta_1 +\dfrac{\epsilon_2}{2}\nabla\left( |\overline{u}_2|^2 \right) = \mu\partial_t\big((\frac{1}{3\delta^2}+\frac{\gamma}{\delta}) \Delta\overline{u}_2+\frac{\gamma}{2}\Delta\overline{u}_1 \big).
\end{array} \right.  \end{equation}
\begin{Rem}
This set of equations had been revealed in~\cite{CCa2}. It corresponds to~\eqref{coefc}, with the choice of parameters: $a_1=-\frac{1}{2}$, $a_2=-\frac{1}{3}$, $b_1=-\frac{1}{3}$. This particular choice of parameters leads to a linearly ill-posed system. That is why it is interesting to obtain, as in \S~\ref{sectbouss}, a larger class of models, allowing linearly well-posed systems.
\end{Rem}
Since this system is a particular case of the Boussinesq/Boussinesq model~\eqref{coefc}, we can apply the Proposition~\ref{consbousscoef}.
\begin{proposition}
 The full system~\eqref{sys} is consistent with~\eqref{BoussBoussLM}, at the precision $\mu^2 C_1$, with \[\begin{array}{r}C_1=C(\frac{1}{h},\beta \big\vert b\big\vert_{W^{1,\infty}H^{s+11/2}},\epsilon_2\big\vert\zeta_2\big\vert_{W^{1,\infty}H^{s+9/2}},\epsilon_1\big\vert\zeta_1\big\vert_{W^{1,\infty}H^{s+7/2}},\big\vert\nabla\psi_1\big\vert_{W^{1,\infty}H^{s+11/2}},\\ \big\vert\nabla\psi_2\big\vert_{W^{1,\infty}H^{s+11/2}}).\end{array}\]
\end{proposition}

\subsubsection{The higher order system}
We now do the same study, without assuming any smallness on $\epsilon_1$, $\epsilon_2$. We plug~\eqref{h1u11} and~\eqref{h2u21} into~\eqref{HO}, and obtain
\begin{equation}\label{HOLM} \left\{ \begin{array}{lr}
\multicolumn{2}{l}{\partial_t h_1 + \epsilon_2\nabla\cdot(h_1\overline{u}_1) =  0,  }\\ 
\multicolumn{2}{l}{\partial_t h_2 +\epsilon_2\nabla\cdot(h_2\overline{u}_2)  =  0, }\\
\multicolumn{2}{l}{\partial_t \overline{u}_1 + \alpha \nabla \zeta_1 +\dfrac{\epsilon_2}{2}\nabla \left(|\overline{u}_1|^2\right)  =  \mu \epsilon_2\nabla\NN_1+\mu\epsilon_2\nabla(\overline{u}_1\cdot \mathcal D_1)+\mu\partial_t \mathcal D_1,} \\ 
\multicolumn{2}{l}{\partial_t \overline{u}_2 + (1-\gamma)\nabla \zeta_2 + \gamma \alpha\nabla \zeta_1 +\dfrac{\epsilon_2}{2}\nabla\left( |\overline{u}_2|^2 \right) = \mu\Big(\partial_t(\gamma\nabla \mathcal{H}+ \mathcal D_2)+\epsilon_2\nabla(\gamma\overline{u}_1\cdot \nabla\mathcal{H}}\\  & +\overline{u}_2\cdot \mathcal D_2+\NN_2+\gamma\NN_1) \Big),
\end{array} \right. \end{equation}
with the notations of Proposition~\ref{uvsphi} and~\eqref{HO}, and when we substitute $\overline{u}_i$ for $\nabla\psi_i$ ($i=1,2$).
\begin{proposition}
The full system~\eqref{sys} is consistent with~\eqref{HOLM}, at the precision $\mu^2 C_1$, with \[\begin{array}{r}C_1=C(\frac{1}{h},\beta \big\vert b\big\vert_{W^{1,\infty}H^{s+11/2}},\epsilon_2\big\vert\zeta_2\big\vert_{W^{1,\infty}H^{s+9/2}},\epsilon_1\big\vert\zeta_1\big\vert_{W^{1,\infty}H^{s+7/2}},\big\vert\nabla\psi_1\big\vert_{W^{1,\infty}H^{s+11/2}},\\ \big\vert\nabla\psi_2\big\vert_{W^{1,\infty}H^{s+11/2}}).\end{array}\]
\end{proposition}
\begin{proof}
 Let $t_0>d/2$ and $s\geq t_0+1/2$. Let $(\zeta_1,\zeta_2,\nabla\psi_1,\nabla\psi_2)$ be a sufficiently smooth solution of~\eqref{sys}, such that~\eqref{h} is satisfied.
We know from Proposition~\ref{consHO} that $(\zeta_1,\zeta_2,\nabla\psi_1,\nabla\psi_2)$ satisfies~\eqref{HO} up to a residual bounded by $\mu^2 C_1$. Then, the estimates~\eqref{h1u11} and~\eqref{h2u21} give that $(\zeta_1,\zeta_2,\overline{u}_1,\overline{u}_2)$ satisfies~\eqref{HOLM} up to a residual of the same order.
\end{proof}

\appendix
\section{Proof of Proposition~\ref{G1}}
\label{PreuveG}

Our proof contains three parts. First we introduce $u$ the correction to the expansion of $\phi_1$ formally obtained in \S~\ref{opexpansion}, and we present the system solved by $u$. Then, we use the elliptic form of the operator to obtain $H^s$ estimates on $u$. Finally, we use these estimates to prove the desired inequalities.

\subsection{System solved by $u$}
We first define the second order correction to the formal expansion:
\[u:=\phi_1-\psi_1+\mu \underbrace{h_1(z-1)\Big(h_1\big(\frac{z+1}{2}\big)\Delta\psi_1-\epsilon_2\nabla\zeta_2\cdot\nabla\psi_1+\nabla\cdot(h_2\nabla\psi_2)\Big)}_{:=\phi^1}.\]
From the computation carried out in \S~\ref{opexpansion}, we know that $u$ satisfies the following equalities:
\begin{eqnarray*}
 \nabla_{X,z}\cdot P_1 \nabla_{X,z}u&=\mu^2\nabla\cdot P^1\nabla \phi^1 &\mbox{ in } \SS^+,\\
 u &=0 &\mbox{ on } \{z=1\},\\
 \partial_n u &= G_2\psi_2 +\mu\nabla\cdot (h_2\nabla\psi_2)+\mu^2(\partial_n^{P^1}\phi^1)& \mbox{ on } \{z=0\}.
\end{eqnarray*}
Moreover, we notice that~\eqref{u2G2} gives $G_2\psi_2 +\mu\nabla\cdot (h_2\nabla\psi_2)=\nabla\cdot V$, with \[V=\mu h_2(\nabla\psi_2-\overline{u}_2).\] Thus, using the definition of $P_1$ in~\eqref{Pi}, we finally have the system
\begin{equation}\label{equ}\left\{\begin{array}{ll}
 \nabla^\mu_{X,z}\cdot P^\mu \nabla^\mu_{X,z}u =\mu^2\ \nabla_{X,z}^\mu\cdot \hh & \mbox{in } \SS^+, \\
u =0 & \mbox{on } \{z=1\}, \\
 \partial_n u =\nabla\cdot V +\mu^2 e_{d+1}\cdot \hh & \mbox{on } \{z=0\},
\end{array}
\right.\end{equation}
where we have introduced the notation $\nabla^\mu_{X,z}:=(\sqrt\mu\nabla^T,\partial_z)^T$, $\hh:=P^\mu \nabla^\mu_{X,z}\phi^1$ and
 \[P^\mu:=\left(\begin{array}{cc}  h_1 I_d & -\sqrt\mu\nabla s_1 \\ -\sqrt\mu{\nabla s_1}^T & \frac{1+\mu|\nabla s_1|^2}{h_1} \end{array}\right).\]

We now give the useful estimates of the right-hand side of the system. It is straightforward to check that 
\begin{equation}\label{esth}\big\|\hh\big\|_{H^{s+1/2,1}}\leq C(\frac{1}{h},\epsilon_1\big\vert \zeta_1\big\vert _{H^{s+3/2}},\epsilon_2\big\vert \zeta_2\big\vert _{H^{s+5/2}},\beta\big\vert b\big\vert _{H^{s+5/2}},\big\vert \nabla\psi_1\big\vert _{H^{s+7/2}},\big\vert \nabla\psi_2\big\vert _{H^{s+7/2}}).\end{equation}
Then, since one has $G_2\psi_2 +\mu\nabla\cdot (h_2\nabla\psi_2)=\nabla\cdot V$, Proposition~\ref{defT} immediately gives
\begin{equation}\label{estdV}\big\vert \nabla\cdot V\big\vert_{H^s}\leq \mu^2 C(\frac{1}{h},\beta\big\vert b\big\vert _{H^{s+7/2}},\epsilon_2\big\vert \zeta_2\big\vert _{H^{s+5/2}},\big\vert \nabla\psi_2\big\vert _{H^{s+7/2}}).\end{equation}
We now seek a $L^2$-estimate of $V=\mu h_2(\nabla\psi_2-\overline{u}_2)$. Using the definition of $\overline{u}_2$ and the mappings defined on \S~\ref{asymptexp}, we obtain easily that 
\[\overline{u}_2-\nabla\psi_2=\int_{-1}^0 \nabla(\t\phi_2-\psi_2)+\nabla s_2\partial_{\t z} \t\phi_2 d\t z,\]
with $\phi_2:(X,\t z)\in \SS^-\mapsto\t \phi_2(X,s_2(X,\t z))$. Then, the method of our proof adapted for the lower fluid (this is done for example in~\cite{Florent}) leads at Step 3 to a $L^2$-estimate on $\nabla_{X,z}^\mu (\t\phi_2-\psi_2)$. We then plug this estimate on the previous equality, deduce the desired estimate on $\big\vert V\big\vert_{L^2}$, and finally get with~\eqref{estdV}: 
\begin{equation}\label{estV}\big\vert V\big\vert_{H^s}\leq \mu^2 C(\frac{1}{h},\beta\big\vert b\big\vert _{H^{s+5/2}},\epsilon_2\big\vert \zeta_2\big\vert _{H^{s+3/2}},\big\vert \nabla\psi_2\big\vert _{H^{s+5/2}}).\end{equation}

\subsection{$H^{s,1}$-estimate ($s\geq 0$) on $u$}\label{Hsestimates} 
We follow the sketch of the proof of Proposition 3 in~\cite{BLS}, which contains five steps.

\paragraph{Step 1} \emph{Coercivity of the operator.}
Since $\zeta_1$, $\zeta_2 \in W^{1,\infty}$ and satisfy~\eqref{h}, we can check (see Proposition 2.3 of~\cite{Lannes2}) that for any $\Theta \in \RR^{d+1}$,
\[\Theta\cdot P^\mu\Theta \geq \frac{1}{k} \big\vert \Theta\big\vert ^2,\]
with $k=\big\| h_1\big\|_\infty+\frac{1}{h}(1+\mu\big\|\nabla s_1\big\|_\infty^2$). The operator is uniformly coercive in $\mu$.

\paragraph{Step 2} \emph{Existence and uniqueness of the solution.}
The result is given by the coercivity of the operator. From the assumptions on $\zeta_1$, $\zeta_2$, $b$, $\psi_1$ and $\psi_2$, we know that $\hh \in H^{s+1/2,1}(\SS^+)^{d+1}$ and $V\in H^{s+1}(\RR^d)$. For $s\geq 1/2$, the proof of Proposition~\ref{flattening} works for the system~\eqref{equ}, so that we know that there exists a unique solution in $H^2(\SS^+)$. We now prove by induction that for $k\in\mathbb N$, 
\begin{equation}\label{smoothing}\hh\in H^{k+1} \mbox{ and } V\in H^k \Longrightarrow u\in H^{k+2}.\end{equation}
We assume that $\hh\in H^{k+2}$ and $V\in H^{k+1}$. We thus know that $u\in H^{k+2}$, so that $v:=\Lambda u \in H^{k+1}\subset H^1$. Hence, $v$ is the classical solution of 
\begin{equation}\left\{\begin{array}{ll}
 \nabla^\mu_{X,z}\cdot P^\mu \nabla^\mu_{X,z}v =\mu^2\ \nabla_{X,z}^\mu\cdot \t\hh & \mbox{in } \SS^+, \\
v =0 & \mbox{on } \{z=1\}, \\
 \partial_n v =\nabla\cdot \Lambda V +\mu^2 e_{d+1}\cdot \partial_x \t\hh & \mbox{on } \{z=0\},
\end{array}
\right.\end{equation}
with $\mu^2\t \hh =\mu^2\Lambda \hh +[\Lambda,P^\mu]\nabla^\mu_{X,z} u$. Thanks to Theorem 6 of~\cite{Lannes3}: for $t_0>\frac{d}{2}$, one has
\[\big\| [\Lambda,P^\mu]\nabla^\mu_{X,z} u\big\|_2 \leq C_{t_0} \big\| \nabla P^\mu\big\|_{H^{t_0}}\big\| \nabla^\mu_{X,z} u\big\|_2, \]
so that $\t\hh\in H^{k+1}$ and $\Lambda V\in H^k$.
The inductive hypotheses are satisfied, so that we know that $v\in H^{k+2}$. Finally, we use the coercivity of the operator (Step 1) with the nth derivative of~\eqref{equ}, and obtain
\[\big\|\partial_z^2 \partial^n u\big\|_2\leq k \big\|\nabla^\mu_{X,z}\cdot P^\mu \nabla^\mu_{X,z}\partial^n u\big\|_2 + \big\|\Delta_X \partial^n u\big\|_2.\]
It follows that $u\in H^{k+3}$, and~\eqref{smoothing} is proved. The interpolation theory leads to the final result: for $s\geq 1/2$, there exists a unique solution $u\in H^{s+3/2}$ of~\eqref{equ}.

\paragraph{Step 3} \emph{$L^2$-estimate on $\nabla_{X,z}^\mu u$.} 
We multiply~\eqref{equ} by $u$, integrate by parts on both sides, and use the boundary conditions to finally obtain
\[\int_\SS\nabla^\mu_{X,z} u\cdot P^\mu\nabla^\mu_{X,z}u=\mu^2\int_\SS\nabla^\mu_{X,z}u\cdot\hh+\int_{\{z=0\}}\nabla u\cdot V .\]
From the coercivity and the Cauchy-Schwarz inequality, we deduce
\[\big\|\nabla^\mu_{X,z}u\big\|_2^2\leq k(\mu^2\big\|\hh\big\|_2\big\|\nabla_{X,z}^\mu u\big\|_2+\big\vert V\big\vert _{H^{1/2}}\big\vert \nabla u_{\vert z=0}\big\vert _{H^{-1/2}}).\]
Then, a trace theorem (see M\'etivier~\cite{Metivier1} pp.23-27) gives
\begin{eqnarray*}
\big\vert \nabla u_{\vert z=0}\big\vert _{H^{-1/2}} &\leq& \mbox{Cst}(\big\|\nabla u\big\|_2+\big\|\Lambda^{-1}\partial_z\nabla u\big\|_2)\\
 &\leq& \mbox{Cst}(\frac{1}{\sqrt\mu}+1)\big\|\nabla_{X,z}^\mu u\big\|_2.
\end{eqnarray*}
This finally gives the estimate
\begin{equation}\label{estHs} \big\|\nabla_{X,z}^\mu u\big\|_2\leq  C(\frac{1}{h},\epsilon_1\big\vert \zeta_1\big\vert _{W^{1,\infty}},\epsilon_2\big\vert \zeta_2\big\vert _{W^{1,\infty}})(\mu^2\big\|\hh\big\|_2+\frac{1+\sqrt\mu}{\sqrt\mu}\big\vert V\big\vert _{H^{1/2}}).\end{equation} 

\paragraph{Step 4} \emph{$L^2$-estimate on $\Lambda^s\nabla^\mu_{X,z}u$ ($s\geq 0$).}
We define $v=\Lambda^s u$. Multiplying~\eqref{equ} by $\Lambda^s$ on both sides, one obtains
\begin{equation}\label{eqv}\left\{\begin{array}{ll}
 \nabla^\mu_{X,z}\cdot P^\mu \nabla^\mu_{X,z}v =\mu^2\ \nabla^\mu_{X,z}\cdot \t\hh & \mbox{in } \RR^d\times(0,1), \\
v =0 & \mbox{on } \{z=1\}, \\
 \partial_n v =\nabla\cdot\Lambda^s V +\mu^2 e_{d+1}\cdot \t\hh & \mbox{on } \{z=0\},
\end{array}
\right.\end{equation}
with $\mu^2\t \hh =\mu^2\Lambda^s \hh -[\Lambda^s,P^\mu]\nabla^\mu_{X,z} u$. We can use Step 3 with $v$ and obtain
\[ \begin{array}{r}\big\|\nabla^\mu_{X,z}v\big\|_2\leq  C(\frac{1}{h},\epsilon_1\big\vert \zeta_1\big\vert _{W^{1,\infty}},\epsilon_2\big\vert \zeta_2\big\vert _{W^{1,\infty}})(\mu^2\big\|\Lambda^s \hh\big\|_2+\big\|[\Lambda^s,P^\mu]\nabla^\mu_{X,z} u\big\|_2\\+\frac{1+\sqrt\mu}{\sqrt\mu}\big\vert V\big\vert _{H^{s+1/2}}).\end{array}\] 
We obtain the commutator estimate thanks to Theorem 6 of~\cite{Lannes3}: for $s>-\frac{d}{2}$ and $t_0>\frac{d}{2}$, one has
\[\big\| [\Lambda^s,f]g\big\|_2 \leq C_{s,t_0} \big\| \nabla f\big\|_{H^{\mathrm{max}\{t_0,s-1\}}}\big\| g\big\|_{H^{s-1}}. \]
In our case, it gives
\[\big\|[\Lambda^s,P^\mu]\nabla^\mu_{X,z} u\big\|_2\leq C_{s,t_0} (\frac{1}{h},\epsilon_1\big\vert \zeta_1\big\vert _{H^{\mathrm{max}\{t_0+2,s+1\}}},\epsilon_2\big\vert \zeta_2\big\vert _{H^{\mathrm{max}\{t_0+2,s+1\}}} )  \big\|\Lambda^{s-1} \nabla^\mu_{X,z} u\big\|_2. \]
We finally get an estimate on $\big\|\Lambda^s\nabla^\mu_{X,z}u\big\|_2$ in terms of $\big\|\Lambda^{s-1}\nabla^\mu_{X,z}u\big\|_2$. Step 3 is the case when $s=0$. By induction, and interpolation when $s\in(0,1)$, we obtain the following relation for all $s\geq 0$:
\begin{equation}\label{estHs1}\begin{array}{r}\big\|\Lambda^s\nabla^\mu_{X,z} u\big\|_2\leq C_{s,t_0} (\frac{1}{h},\epsilon_1\big\vert \zeta_1\big\vert _{H^{\mathrm{max}\{t_0+2,s+1\}}},\epsilon_2\big\vert \zeta_2\big\vert _{H^{\mathrm{max}\{t_0+2,s+1\}}} ) (\mu^2\big\|\Lambda^s \hh\big\|_2\\+\frac{1+\sqrt\mu}{\sqrt\mu}\big\vert V\big\vert _{H^{s+1/2}}).\end{array}\end{equation}

\paragraph{Step 5} \emph{$L^2$-estimate ($s\geq 0$) on $\Lambda^s\partial_z\nabla^\mu_{X,z}u$.}
The equation~\eqref{equ} gives the formula
\[\begin{array}{r}\frac{1+\mu\big\vert \nabla s_1\big\vert ^2}{h_1}\partial_z^2u=\mu^2 \nabla_{X,z}^\mu\cdot \hh-\mu\nabla\cdot(h_1\nabla u-\nabla s_1 \partial_z u)+\mu\partial_z(\nabla s_1\cdot\nabla u)\\-\partial_z(\frac{1+\mu\big\vert \nabla s_1\big\vert ^2}{h_1})\partial_z u,\end{array}\]
from which we deduce
\[\big\|\Lambda^s \partial_z^2 u\big\|_2\leq C(\frac{1}{h},\epsilon_1\big\vert \zeta_1\big\vert _{W^{1,\infty}},\epsilon_2\big\vert \zeta_2\big\vert _{W^{1,\infty}}) (\mu^2\big\|\Lambda^s\nabla^\mu_{X,z}\cdot \hh\big\|_2+ \sqrt\mu\big\|\Lambda^{s+1} \nabla^\mu_{X,z} u\big\|_2) .\]
Thus, we have the estimate
\[\big\|\Lambda^s \partial_z\nabla^\mu_{X,z} u\big\|_2\leq C(\frac{1}{h},\epsilon_1\big\vert \zeta_1\big\vert _{W^{1,\infty}},\epsilon_2\big\vert \zeta_2\big\vert _{W^{1,\infty}})(\mu^2\big\|\Lambda^s\nabla^\mu_{X,z}\cdot \hh\big\|_2+\sqrt\mu \big\|\Lambda^{s+1} \nabla_{X,z}^\mu u\big\|_2),\]
and Step 4 allows us to conclude
\[\begin{array}{r}\big\|\Lambda^s \partial_z\nabla^\mu_{X,z} u\big\|_2\leq C_{s,t_0}(\frac{1}{h},\epsilon_1\big\vert \zeta_1\big\vert _{H^{\mathrm{max}\{t_0+2,s+2\}}},\epsilon_2\big\vert \zeta_2\big\vert _{H^{\mathrm{max}\{t_0+2,s+2\}}})(\mu^2\big\| \hh\big\|_{H^{s+1,1}}\\+ \big\vert V\big\vert _{H^{s+3/2}}).\end{array}\]

\subsection{Proof of the inequalities}
To obtain the first estimate, we remark that
\[G_1(\psi_1,\psi_2)+\mu (\mathcal{A}_1+\mathcal{A}_2)=\partial_n u_{|z=1}-\mu^2 u_0,\]
with $u_0:=|\epsilon_1\nabla\zeta_1|^2\Big(h_1\Delta\psi_1-\epsilon_2\zeta_2\cdot\nabla\psi_1+\nabla\cdot(h_2\nabla\psi_2) \Big)$. It is straightforward to check that
\[ \big\vert u_0 \big\vert_{H^s}\leq C(\frac{1}{h},\beta \big\vert b\big\vert_{H^{s+1}},\epsilon_2\big\vert\zeta_2\big\vert_{H^{s+1}},\epsilon_1\big\vert\zeta_1\big\vert_{H^{s+1}},\big\vert\nabla\psi_1\big\vert_{H^{s+2}},\big\vert\nabla\psi_2\big\vert_{H^{s+2}}),\]
so that we just have to bound  $\big\vert\partial_n u_{|z=1}\big\vert_{H^s}$.
We now use the trace theorem to get
\begin{eqnarray}
 \big\vert\partial_n u_{|z=1}\big\vert_{H^s} &\leq&  C(\frac{1}{h},\epsilon_1\big\vert\zeta_1\big\vert_{W^{1,\infty}},\epsilon_2\big\vert\zeta_2\big\vert_{W^{1,\infty}}) (\mu\big\vert\nabla u_{\vert z=1} \big\vert_{H^s}+\big\vert\partial_z u_{\vert z=1} \big\vert_{H^s}) \nonumber\\
&\leq& C(\frac{1}{h},\epsilon_1\big\vert\zeta_1\big\vert_{W^{1,\infty}},\epsilon_2\big\vert\zeta_2\big\vert_{W^{1,\infty}})(\sqrt\mu\big\|\nabla^\mu_{X,z} u\big\|_{H^{s+1/2,0}}\nonumber\\
&&+\big\|\partial_z\nabla_{X,z}^\mu u\big\|_{H^{s-1/2,0}} ).
\end{eqnarray}
The estimates obtained in Steps 4 and 5, together with~\eqref{esth} and~\eqref{estV}, give immediately the desired result.

To obtain the second estimate, one has to carry on the proof with the higher order approximate solution obtained in \S~\ref{opexpansion}:
\begin{equation*}
\t u:=\phi_1-\phi_1^{app,2},
\end{equation*}
and one would obtain the estimates exactly as above. We omit this technical step.

 %\tableofcontents
\paragraph{Acknowledgments}
Ce travail a bénéficié d'une aide de l'Agence Nationale de la Recherche portant la référence ANR-08-BLAN-0301-01.
\bibliographystyle{siam}
\bibliography{Biboceano}

\end{document}